\documentclass[11pt]{article}
\usepackage{amsthm,amsmath,amssymb}
\usepackage{natbib}
\usepackage{multirow}
\usepackage[pdftex]{graphicx}
\usepackage{subfigure}
\usepackage{makecell}
\usepackage{booktabs}
\usepackage{array}
\usepackage{fullpage}
\usepackage{url}
\usepackage{algorithm}
\usepackage{algorithmic}
\usepackage{bm}
\usepackage{smile}
\usepackage{mathtools}
\usepackage{wrapfig}
\usepackage{lipsum}
\usepackage{mathrsfs}
\usepackage{dsfont}
\usepackage{verbatim}
\usepackage{graphicx}
\usepackage{enumitem}

\usepackage{natbib}
\usepackage{multirow}

\usepackage{appendix}
\allowdisplaybreaks

\newcolumntype {Q}{>{$\displaystyle}l<{$}}
\newcolumntype {A}{>{$}c <{$}}

\def\tr{\mathop{\text{tr}}\kern.2ex}

\def\P{{\mathbb P}}
\def\E{{\mathbb E}}

\def\Z{{\mathbb Z}}

\def\Sb{{\mathbf S}}
\def\T{{\intercal}}

\def\card{\mathop{\text{card}}}

\long\def\comment#1{}

\def\tr{\mathop{\text{Tr}}}

\providecommand{\norm}[1]{\vvvert#1\vvvert}

\newcommand{\bel}{\begin{eqnarray}\label}
\newcommand{\eel}{\end{eqnarray}}
\newcommand{\bes}{\begin{eqnarray*}}
\newcommand{\ees}{\end{eqnarray*}}

\def\reals{{\mathbb{R}}}
\def\R{{\reals}}
\def\T{{\sf T}}

\let\bar\overline

\usepackage[usenames,dvipsnames,svgnames,table]{xcolor}
\usepackage[colorlinks=true,
            linkcolor=blue,
            urlcolor=blue,
            citecolor=blue]{hyperref}
            
\numberwithin{equation}{section}
\numberwithin{theorem}{section}
\numberwithin{corollary}{section}
\numberwithin{asmp}{section}
\numberwithin{definition}{section}

\begin{document}

\setlength{\abovedisplayskip}{5pt}
\setlength{\belowdisplayskip}{5pt}
\setlength{\abovedisplayshortskip}{5pt}
\setlength{\belowdisplayshortskip}{5pt}

\title{\LARGE Moment bounds for large autocovariance matrices under dependence}

\author{Fang Han\thanks{Department of Statistics, University of Washington, Seattle, WA 98195, USA; e-mail: {\tt fanghan@uw.edu}}~~~and~ Yicheng Li\thanks{Department of Statistics, University of Washington, Seattle, WA 98195, USA; e-mail: {\tt yl83@uw.edu}}}

\date{}

\maketitle
\begin{abstract}
The goal of this paper is to obtain expectation bounds for the deviation of large sample autocovariance matrices from their means under weak data dependence. While the accuracy of covariance matrix estimation corresponding to independent data has been well understood, much less is known in the case of dependent data. We make a step towards filling this gap, and establish deviation bounds that depend only on the parameters controlling the ``intrinsic dimension" of the data up to some logarithmic terms. Our results have immediate impacts on high dimensional time series analysis, and we apply them to high dimensional linear VAR($d$) model, vector-valued ARCH model, and a model used in \cite{banna2016bernstein}. 
\end{abstract}



\section{Introduction}\label{sec:intro}
Consider a sequence of $p$-dimensional mean-zero random vectors $\{\bY_t\}_{t\in\Z}$ and a size-$n$ fraction $\{\bY_i\}_{i=1}^n$ of it. This paper aims to establish moment bounds for the spectral norm deviation of lag-$m$ autocovariances of $\{\bY_i\}_{i=1}^n$, $\hat\bSigma_m:=(n-m)^{-1}\sum_{i=1}^{n-m}\bY_i\bY_{i+m}^\T$, from their mean values. 

A first result at the origin of such problems concerns product measures, with $m=0$ and $\{\bY_i\}_{i=1}^n$ independent and identically distributed (i.i.d.). For this, \cite{rudelson1999random} derived a bound on $\E\norm{\hat\bSigma_0-\E\hat\bSigma_0}$, 
where 
$\norm{\cdot}$ represents the spectral norm for matrices. The technique is based on symmetrization and the derived maximal inequality is a consequence of a concentration inequality on a ``symmetrized" version of $p\times p$ symmetric and deterministic matrices, $\Ab_1,\ldots,\Ab_n$ (cf. \cite{oliveira2010sums}). That is, for any $x\geq0$,
\begin{align}\label{eq:rudelson2}
\P\Big(\Big\|\sum_{i=1}^n\epsilon_i\Ab_i\Big\|\geq x \Big)\leq 2p\exp\{-x^2/(2\sigma^2)\},~~\sigma^2:=\Big\|\sum_{i=1}^n\Ab_i^2\Big\|,
\end{align}
where $\{\epsilon_i\}_{i=1}^n$ are independent and taking values $\{-1,1\}$ with equal probability. The applicability of this technique then hinges on the assumption that the data are i.i.d.. 

Later, \cite{vershynin2010introduction}, \cite{srivastava2013covariance}, \cite{mendelson2014singular}, \cite{lounici2014high}, \cite{bunea2015sample}, \cite{tikhomirov2017sample}, among many others, derived different types of deviation bounds for $\hat\bSigma_0$ under different distributional assumptions. For example,  \cite{lounici2014high} and \cite{bunea2015sample} showed that, for such $\{\bY_i\}_{i=1}^n$ that are subgaussian and i.i.d., 
\begin{align}\label{eq:key0}
\E\norm{\hat\bSigma_0-\bSigma_0}\leq C\norm{\bSigma_0}\Big\{\sqrt{\frac{r(\bSigma_0)\log (ep)}{n}}+\frac{r(\bSigma_0)\log (ep)}{n} \Big\}.
\end{align}
Here $C>0$ is a universal constant, $\bSigma_0:=\E\bY_1\bY_1^\T$, and $r(\bSigma_0):=\tr(\bSigma_0)/\norm{\bSigma_0}$ is termed the ``effective rank" \citep{vershynin2010introduction} where $\tr(\Xb) := \sum_{i = 1}^p\Xb_{i,i}$ for any real $p\times p$ matrix $\Xb$.

Statistically speaking, Equation \eqref{eq:key0} is of rich implications. For example, combining \eqref{eq:key0} with Davis-Kahan inequality \citep{davis1970rotation} suggests that the principal component analysis (PCA), a core statistical method whose aim is to recover the leading eigenvectors of $\bSigma_0$, could still produce consistent estimators even if the dimension $p$ is much larger than the sample size $n$, as long as the ``intrinsic dimension" of the data, quantified by $r(\bSigma_0)$, is small enough. See Section 1 in \cite{han2018eca} for more discussions on the statistical performance of PCA in high dimensions.


The main goal of this paper is to give extensions of the deviation inequality \eqref{eq:key0}  to large autocovariance matrices, where the matrices are constructed from a high dimensional structural time series. Examples of such time series include linear vector autoregressive model of lag $d$ (VAR($d$)), vector-valued autoregressive conditionally heteroscedastic (ARCH) model, and a model used in \cite{banna2016bernstein}. The main result appears below as Theorem \ref{thm:subgaummb}, and is nonasymptotic in its nature. 
This result will have important consequences in high dimensional time series analysis. For example, it immediately yields new analysis for estimating large covariance matrix \citep{chen2013covariance},  a new proof of consistency for Brillinger's PCA in the frequency domain (cf. Chapter 9 in \cite{brillinger2001time}), and we envision that it could facilitate a new proof of consistency for the PCA procedure proposed in \cite{chang2018principal}.

The rest of the paper is organized as follows. Section \ref{sec:thm} characterizes the settings and gives the main concentration inequality for large autocovariance matrices. In Section \ref{sec:app}, we present applications of our results to some specific time series models. Proofs of the main results are given in Section \ref{sec:proof}, with more relegated to an appendix.

\section{Main results}\label{sec:thm}
We first introduce the notation that will be used in this paper. Without further specification, we use bold, italic lower case alphabets to denote vectors, e.g., $\bu = (u_1,\cdots,u_p)^\T$ as a $p$-dimensional real vector, and $\norm{\bu}_2$ as its vector $L_2$ norm. We use bold, upper case alphabets to denote matrices, e.g., $\Xb = (X_{i,j})$ as a $p \times p$ real matrix, and $\Ib_p$ as the $p\times p$ identity matrix. Throughout the paper, let $c, c', C, C', C''$ be generic universal positive constants, whose actual values may vary at different locations. 
For any two sequences of positive numbers $\{a_n\}, \{b_n\}$, we denote $a_n=O(b_n)$ if there exists an universal constant $C$ such that $a_n \leq Cb_n$ for all $n$ large enough. We write $a_n\asymp b_n$ if both $a_n=O(b_n)$ and $b_n=O(a_n)$ hold.

Consider a time series $\{\bY_t\}_{t\in\Z}$ of $p$-dimensional real entries $\bY_t\in\reals^p$ with $\R,\Z$ denoting the sets of real and integer numbers respectively. In the sequel, the considered time series does not need to be stationary nor centered, and we are focused on a size-$n$ fraction of it. Without loss of generality, we denote this fraction to be $\{\bY_i\}_{i=1}^n$. 

As described in the introduction, the case of independent $\{\bY_i\}_{i=1}^n$ has been discussed in depth in recent years. We are interested here in the time series setting, and our main emphasis will be to describe nontrivial but easy to verify cases for which Inequality \eqref{eq:key0} still holds. The following four assumptions are accordingly made, with the notations that 
\[
\mathbb{S}^{p-1} := \{\bx \in \R^p: \norm{\bx}_2 = 1\}, \quad \bar{\mathbb{S}}^{p-1} := \{\bx \in \R^p: |x_1| = \dots = |x_p| = 1\},
\] 
and 
\[
\norm{X}_{L(p)}:=(\E|X|^p)^{1/p}, \quad \norm{X}_{\psi_2} := \inf \{k \in (0, \infty): \E[\exp\{(|X|/k)^2\} -1]\leq 1\} 
\]
for any random variable $X$.

\begin{enumerate}[label=\textbf{(A\arabic*)}]
	
		\item\label{A1} Define 
	\begin{align*}
	&\kappa_1 := \sup_{t \in \Z} \sup_{\bu\in \mathbb{S}^{p-1}} \norm{\bu^\T \bY_t}_{\psi_2} < \infty, ~~~ \kappa_* := \sup_{t \in \Z} \sup_{\bv\in \bar{\mathbb{S}}^{p-1}} \norm{\bv^\T \bY_t}_{\psi_2} < \infty.
	\end{align*}
	Note that $\kappa_1$ is the supremum taken over vectors in the unit hypersphere, while $\kappa_*$ is the supremum taken over vectors in the discrete hypercube.

	\item \label{A2} Assume that there exist some constants $\gamma_1, \gamma_2, \epsilon >0$ such that for any integer $j$, there exists a sequence of random vectors $\{\tilde{\bY}_t\}_{t > j}$ which is independent of $\sigma(\{\bY_t\}_{t \leq j})$, identically distributed as $\{\bY_t\}_{t > j}$, and for any integer $k \geq j+1$,
	\[
	\norm{\norm{\bY_k-\tilde{\bY}_k}_2 }_{L(1+\epsilon)} \leq \gamma_1\kappa_1 \exp\{-\gamma_2(k-j-1)\}.
	\]

	\item \label{A3} Assume that there exist some constants $\gamma_3, \gamma_4, \epsilon >0$ such that for any integer $j$, there exists a sequence of random vectors $\{\tilde{\bY}_t\}_{t > j}$ which is independent of $\sigma(\{\bY_t\}_{t \leq j})$, identically distributed as $\{\bY_t\}_{t > j}$, and for any integer $k \geq j+1$,
	\[
	\sup_{\bu\in \mathbb{S}^{p-1}}\norm{(\bY_k-\tilde{\bY}_k)^\T\bu}_{L(1+\epsilon)}\leq \gamma_3\kappa_1 \exp\{-\gamma_4(k-j-1)\}.
	\]
	
	\item\label{A4} Assume there exists an universal constant $c > 0$ such that, for all $t \in \Z$ and for all $\bu \in \R^{p}$, $\norm{\bu^{\T}  \bY_t}_{\psi_2}^2\leq  c\E(\bu^{\T}\bY_t)^2$ .
\end{enumerate}

Two observations are in order. We first define a generalized ``effective rank" as follows:
	$$r_* := \kappa_*^2/\kappa_1^2.$$
It is easy to see the close relationship between $r_*$ and the effective rank highlighted in \eqref{eq:key0}. As $\bY_t\sim N(\zero,\bSigma_0)$, $\kappa_1^2$ and $\kappa_*^2$ scale at the same orders of $\norm{\bSigma_0}$ and $\tr(\bSigma_0)$, and the same observation applies to all subgaussian distributions with the additional condition \ref{A4}, which is identical to Assumption 1 in \cite{lounici2014high}. As a matter of fact, $r_*$ could be considered as a natural generalized version of $r(\bSigma_0)$ without these additional assumptions, and is used in our main theorem.

Secondly, we note that Assumptions \ref{A2} and \ref{A3} are characterizing the intrinsic coupling property of the sequence. 
In practice, such couples can be constructed from time to time. Consider, for example, the following causal shift model,
\[
\bY_t=H_t(\xi_t,\xi_{t-1},\xi_{t-2},\ldots),
\]
where $\{\xi_t\}_{t\in\Z}$ consists of independent elements with values in a measurable space $\cX$ and $H_t:\cX^{\Z^+}\to \reals^p$ is a vector-valued function. Then it is natural to consider
\[
\tilde\bY_t=H_t(\xi_t,\ldots\xi_{j+1},\tilde\xi_{j},\tilde\xi_{j-1},\ldots)
\]
for an independent copy $\{\tilde\xi_t\}_{t\in\Z}$ of $\{\xi_t\}_{t\in\Z}$. 


The following is the main result of this paper. 


\begin{theorem}[Proof in Section \ref{subsec:thm2.1}]\label{thm:subgaummb}
	Let $\{\bY_t\}_{t \in \Z}$ be a sequence of random vectors satisfying Assumptions \ref{A1}-\ref{A3} and recall $r^*=\kappa_*^2/\kappa_1^2$. Assume $\gamma_1 = O(\sqrt{r^*})$ and $\gamma_3 = O(1)$. Then, for any integer $n \geq 2$ and $0\leq m \leq n-1$, we have
	\begin{align}\label{eq:main}
	\E\norm{\hat{\mathbf{\Sigma}}_m - \E\hat{\mathbf{\Sigma}}_m} \leq C\kappa_1^2\Big\{\sqrt{\frac{r_*\log ep}{n-m}} + \frac{r_*\log ep (\log np)^3}{n-m}  \Big\}
	\end{align}
	for some constant $C$ only depending on $\epsilon, m, \gamma_2, \gamma_4$. If in addition, $\{\bY_t\}_{t \in \Z}$ is a second-order stationary sequence of mean-zero random vectors and Assumption \ref{A4} holds, then
	$$\E\norm{\hat{\mathbf{\Sigma}}_m - \E\hat{\mathbf{\Sigma}}_m} \leq C'\norm{\mathbf{\Sigma}_0}\Big\{\sqrt{\frac{r(\bSigma_0)\log ep}{n-m}} + \frac{r(\bSigma_0)\log ep (\log np)^3}{n-m}  \Big\}$$
	for some constant $C'$ only depending on $\epsilon,c, m, \gamma_2, \gamma_4$. 
\end{theorem}

We first comment on the temporal correlatedness conditions, Assumptions \ref{A2} and \ref{A3}. We note that they correspond exactly to the $\delta$-measure of dependence  introduced in Chapter 3 of \cite{dedecker2007book}, for the sequence $\{\bY_t\}_{t\in\Z}$ and $\{\bu^\T\bY_t\}_{t\in\Z}$ respectively.  In addition, as will be seen soon, our measure of dependence is also very related to the $\tau$-measure introduced in \cite{dedecker2004coupling}. In particular, ours is usually stronger than, but as $\epsilon\to0$,  reduces to the $\tau$-measure. Lastly, our conditions are also quite connected to the functional dependence measure in \cite{wu2005nonlinear}, on which many moment inequalities in real space have been established (cf. \cite{liu2013probability} and \cite{wu2016performance}). However, it is still unclear if a similar matrix Bernstein inequality could be developed under Weibiao Wu's functional dependence condition. 

Secondly, we note that one is ready to verify that Inequality \eqref{eq:main} gives the exact control of the deviation from the mean. Actually, Inequality \eqref{eq:main} is nearly a strict extension of the results in Lounici \citep{lounici2014high} and Bunea and Xiao \citep{bunea2015sample} to weak data dependence up to some logarithmic terms. This extension is achieved by applying Theorem \ref{thm:bern}, a concentration inequality for a sequence of weakly dependent random matrices. Theorem \ref{thm:bern} is an extension of the Bernstein-type inequality for real-valued weakly dependent random variables derived in \cite{merlevede2011bernstein} to dependent random matrices, and is a slight extension of the Bernstein-type inequality for a sequence of $\beta$-mixing random matrices derived in \cite{banna2016bernstein}. In some applications, especially those in high dimensions, verifying the weak dependence condition in Theorem \ref{thm:bern} is more straightforward than verifying the $\beta$-mixing condition in Theorem 1 in \cite{banna2016bernstein}. The details of the weak dependence condition will be introduced in Section \ref{subsec:thm2.1}, and Theorem \ref{thm:bern} will be proved in the Appendix.

Admittedly, it is still unclear if Inequality \eqref{eq:main} could be further improved under the given conditions. Recently, in a remarkable series of papers \citep{koltchinskii2014concentration,koltchinskii2017new,koltchinskii2017normal}, Koltchinskii and Lounici showed that, for subgaussian independent data, the extra multiplicative $p$ term on the righthand side of Inequality \eqref{eq:main} could be further removed. The proof rests on Talagrand's majorizing measures \citep{talagrand2014upper} and a corresponding maximal inequality due to Mendelson \citep{mendelson2010empirical}. 
In the most general case, to the authors' knowledge, it is still unknown if Talagrand's approach could be extent to weakly dependent data, although we conjecture that, under stronger temporal dependence (e.g., geometrically $\phi$-mixing) conditions, it is possible to recover Koltchinskii and Lounici's result without resorting to the matrix Bernstein inequality in the proof of Theorem \ref{thm:subgaummb}. 

Nevertheless, we make a first step towards eliminating these logarithmic terms via the following theorem. It shows, when assuming a Gaussian sequence is observed, one could further tighten the upper bound in Inequality \eqref{eq:main} by removing all logarithm factors. The obtained bound is thus tight in view of Theorem 2 in \cite{lounici2014high} and Theorem 4 in \cite{koltchinskii2014concentration}.

\begin{theorem}[Proof in Section \ref{subsec:thm2.2}]\label{thm: gaummb}
	Let $\{\bY_t\}_{t \in \Z}$ be a stationary mean-zero Gaussian sequence that satisfies Assumptions \ref{A2}-\ref{A3} with $\gamma_1 = O(\sqrt{r(\bSigma_0)})$, $\gamma_3 = O(1)$, and $\epsilon >1$. Then, for any integer $n \geq 2$ and $0\leq m \leq n-1$,
	$$ \E\norm{\hat{\mathbf{\Sigma}}_m- \mathbf{\Sigma}_m} \leq C\norm{\mathbf{\Sigma}_0}\Big(\sqrt{\frac{r(\bSigma_0)}{n-m}} + \frac{r(\bSigma_0)}{n-m} \Big) $$
	for some constant $C>0$ only depending on $ \epsilon,m, \gamma_2, \gamma_4$.
\end{theorem}

In a related track of studies, \cite{bai1993limit}, \cite{srivastava2013covariance}, \cite{mendelson2014singular}, and \cite{tikhomirov2017sample}, among many others, explored the optimal scaling requirement in approximating a large covariance matrix for heavy-tailed  data. For instance, for i.i.d. data and as $\bSigma_0$ is identity, Bai and Yin \citep{bai1993limit} showed that $\norm{\hat\bSigma_0-\bSigma_0}$ will converge to zero in probability as long as $p/n\to 0$ and 4-th moments exist. Some recent developments further strengthen the moment requirement. 
These results cannot be compared to ours. In particular, our analysis is focused on characterizing the role of ``effective rank", a term of strong meanings in statistical implications and a feature that cannot be captured using these alternative procedures. 

\section{Applications}\label{sec:app}

In this section, we examine the validity of Assumptions \ref{A1}-\ref{A4} in Section \ref{sec:thm} under three models, a stable VAR($d$) model, a model proposed by \cite{banna2016bernstein}, and an ARCH-type model. One shall be aware of examples that are of VAR($d$) or ARCH-type structures but are not $\alpha$- or $\beta$-mixing 
(cf. \cite{andrews1984non}).


We first consider such $\{\bY_t\}_{t \in \Z}$ that is a random sequence generated from VAR($d$) model, i.e., 
$$\bY_t = \Ab_1\bY_{t-1}+ \dots+ \Ab_d\bY_{t-d}+ \bE_t,$$
where $\{\bE_t\}_{t \in \Z}$ is a sequence of independent vectors such that for all $t \in \Z$ and $\bu \in \R^p$, $\norm{\bu^\T\bE_t}_{\psi_2}  \leq  c'\norm{\bu^\T\bE_t}_{L(2)}$ for some universal constant $c'>0$. In addition, assume $\sup_{t \in \Z}\sup_{\bu \in \mathbb{S}^{p-1}} \norm{\bu^\T\bE_t}_{\psi_2} <D_1$ for some universal positive constant $D_1<\infty$,  $\norm{\Ab_k} \leq a_k <1$ for all $1\leq k\leq d$, and $\sum_{k = 1}^d a_k <1$, where $\{a_k\}_{k = 1}^d, d$ are some universal constants. 

Under these conditions, we have the following theorem.
\begin{theorem}[Proof in Section \ref{subsec:pfsec3}]\label{ex1}
The above $\{\bY_t\}_{t \in \Z}$ satisfies Assumptions \ref{A1}-\ref{A4} with 
\[
\gamma_1= C(\kappa_*/\kappa_1)(\norm{\bar{\Ab}}/\rho_1)^K, \gamma_2=\log(\rho_1^{-1}), \gamma_3= C'd(\norm{\bar{\Ab}}/\rho_1)^K, \gamma_4 = \log(\rho_1^{-1}).
\]
Here we denote
	\[
\bar{\Ab} := \begin{bmatrix}
a_1      & a_2 & \dots & a_{d-1} & a_d \\
1      & 0 & \dots& 0 & 0\\
\hdotsfor{5} \\
0     & 0& \dots& 1& 0
\end{bmatrix},
\]
$\rho_1$ is a universal constant such that $\rho(\bar{\Ab}) < \rho_1 <1$ whose existence is guaranteed by the assumption that $\sum_{k = 1}^d a_k <1$ (cf. Lemma \ref{lem:rhoA} in Section \ref{sec:proof}), $K$ is some constant only depending on $\rho_1$, and $C,C'>0$ are some constants only depending on $\epsilon$.
\end{theorem}

We secondly consider the following time series generation scheme whose corresponding matrix version has been considered by Banna, Merlev\`{e}de, and Youssef \citep{banna2016bernstein}. In detail, let $\{\bY_t\}_{t \in \Z}$ be a random sequence generated by
	$$\bY_t = W_t\bE_t,$$
	where $\{\bE_t\}_{t \in \Z}$ is a sequence of independent random vectors independent of $\{W_t\}_{t \in \Z}$ such that for all $t\in \Z$ and $\bu\in \R^p$,
	$\norm{\bu^\T\bE_t}_{\psi_2}  \leq c' \norm{\bu^\T\bE_t}_{L(2)}$ for some universal constant $c'>0$. In addition, we assume
	 \[
	 \sup_{t \in \Z}\sup_{\bu \in \mathbb{S}^{p-1}} \norm{\bu^\T\bE_t}_{\psi_2} \leq  \kappa_1'~~~{\rm and}~~~  \sup_{t \in \Z}\sup_{\bv \in \bar{\mathbb{S}}^{p-1}}\norm{\bv^\T\bE_t}_{\psi_2}  \leq \kappa_*'  
	 \]
	for some constants $0 < \kappa_1',  \kappa_*' < \infty$, $\{W_t\}_{t \in \Z}$ is a sequence of uniformly bounded $\tau$-mixing random variables such that $\max_{t\in \Z} |W_t| \leq \kappa_W$, and 
	$$\tau(k; \{W_t\}_{t \in \Z}, |\cdot|) \leq\kappa_W\gamma_5\exp\{-\gamma_6(k-1)\}$$
	for some constants $0 < \gamma_5, \gamma_6, \kappa_W <\infty$ (see, Appendix Section \ref{sec:tau} for a detailed introduction to the $\tau$-mixing random variables).

\begin{theorem}[Proof in Section \ref{subsec:pfsec3}]\label{ex2}
 The above $\{\bY_t\}_{t \in \Z}$ satisfies Assumptions \ref{A1}-\ref{A4} with 
 \[
 \gamma_1=C\kappa_*' \kappa_W \gamma_5^{\frac{1}{1+\epsilon}}/\kappa_1, \gamma_2=\gamma_6/(1+\epsilon), \gamma_3=C'\kappa_1' \kappa_W \gamma_5^{\frac{1}{1+\epsilon}}/\kappa_1, \gamma_4 = \gamma_6/(1+\epsilon) 
 \]
 for some constants $C, C' >0$ only depending on $\epsilon$.
\end{theorem}

Lastly, we consider an vector-valued ARCH-model with $\{\bY_t\}_{t \in \Z}$ being a random sequence generated by
	$$\bY_t = \Ab \bY_{t-1}+ H(\bY_{t-1})\bE_t,$$
	where $H:\R^p \rightarrow \R^{p\times p}$ is a matrix-valued function and $\{\bE_t\}_{t \in \Z}$ is a sequence of independent random vectors such that
		$$ \sup_{t \in \Z}\sup_{\bu \in \mathbb{S}^{p-1}} \norm{\bu^\T\bE_t}_{\psi_2} \leq  \kappa_1'~~~{\rm and}~~~ \sup_{t \in \Z}\sup_{\bv \in \bar{\mathbb{S}}^{p-1}}\norm{\bv^\T\bE_t}_{\psi_2}  \leq \kappa_*'  $$
	for some constants $0 < \kappa_1',  \kappa_*' < \infty$. Assume further that $\norm{\Ab} \leq a_1$ and the function $H(\cdot)$ satisfies
	\begin{align*}
	&\sup_{\bu,\bv\in\reals^p}\norm{H(\bu) - H(\bv)}  \leq \frac{a_2}{\kappa_{*}'}\norm{\bu - \bv}_2
	\end{align*}
	for some universal constant $a_1<1, a_2>0$ such that $a_1 + a_2<1$.

\begin{theorem}[Proof in Section \ref{subsec:pfsec3}]\label{ex3}
 If the above $\{\bY_t\}_{t \in \Z}$ satisfies Assumption \ref{A1}, it satisfies Assumptions \ref{A2}-\ref{A3} with 
 \[
 \gamma_1=C\kappa_*/\kappa_1, \gamma_2=-\log(a_1+a_2), \gamma_3=C'\max(\kappa_*\kappa_{1}'/\kappa_1\kappa_{*}', 1), \gamma_4 = \log(a_1+a_2)^{-1}
 \]
  for some constants $C, C'>0$ only depending on $\epsilon$. If we further assume the above $\{\bY_t\}_{t \in \Z}$ to be a stationary sequence and $\sup_{\bu\in\reals^p}\norm{H(\bu)}<D_2$ for some universal constant $D_2 <\infty$, then $\{\bY_t\}_{t \in \Z}$ satisfies Assumption \ref{A1}.
\end{theorem}

\section{Proofs}\label{sec:proof}
\subsection{Proof of Theorem \ref{thm:subgaummb}}\label{subsec:thm2.1}

\begin{proof}[Proof of Theorem \ref{thm:subgaummb}]
	 The proof depends mainly on the following tail probability bound of deviation of the sample covariance from its mean.
	
\begin{proposition}[Proof in Section \ref{subsec:thm2.1}]\label{prop:tailbd}
	Let $\{\bY_t\}_{t \in \Z}$ be a sequence of random vectors satisfying \ref{A1}-\ref{A3}. For any integer $n\geq 2$, integer $0\leq m\leq n-2$ and real number $0< \delta \leq 1$, define 
	\[
	M_\delta := C\max\Big\{
	\Big( \frac{\kappa_*}{\kappa_1}\Big)^2 \log \frac{n-m}{\delta}, \Big( \frac{\kappa_*}{\kappa_1} \Big)^2, \frac{2\kappa_*\gamma_1}{\kappa_1}
	\Big\}.
	\]
	Then for any $x \geq 0$,
	\begin{align*}
		\P[\norm{\hat{\mathbf{\Sigma}}_m - \E\hat{\mathbf{\Sigma}}_m} \geq \kappa_1^2\{x +\sqrt{\delta/(n-m)}\}]
	\leq  2p\exp\bigg\{-\frac{C'(n-m)^2x^2}{A_1(n-m) + A_2M_\delta^2 + A_3(n-m)xM_{\delta}} \bigg\}+ \delta, 
	\end{align*}
	with \begin{align*}
	&A_1 :=\frac{ \{
		\kappa_*\gamma_1/\kappa_1 + (\kappa_*/\kappa_1)^2(\gamma_3 + 2m+1) + 2m+1
		\}}{1-\exp\{-\min( \frac{5+\epsilon}{6\epsilon+10}\gamma_2, \gamma_4 )\}}, \hspace{1cm}A_2 := \frac{453^2}{\gamma_2},\\ 
	&A_3 := \frac{2\log (n-m)}{\log 2}\max\bigg\{1, 8m + \frac{48\log {(n-m)p}}{\gamma_2} \bigg\}
	\end{align*}
	for some constants $C,C' >0$ only depending on $\epsilon$.
\end{proposition}
	
	Without loss of generality, let $m = 0$. Taking $x = \sqrt{\frac{r_*\log ep}{n} t}$, $\delta = x^{-\gamma}$ for some $\gamma>1$, $\gamma_1 = O(\sqrt{r_*})$, and  $\gamma_3 = O(1)$ in Proposition \ref{prop:tailbd}, we obtain 
\begin{align*}
\P\Big(\norm{\hat{\mathbf{\Sigma}}_0 - \E\hat{\mathbf{\Sigma}}_0 } \geq C_1\kappa_1^2\sqrt{\frac{r_*\log ep}{n} t}\Big) \leq 2p\exp\bigg[
-\frac{ C_2 (\log ep) t/\{\log(\sqrt{\frac{r_*\log ep}{n} t})\}^2  }{1 + \frac{r_*(\log n)^2}{n} + \sqrt{\frac{r_*\log ep}{n}t} (\log np)^3}
\bigg] + x^{-\gamma}
\end{align*}
for some constants $C_1, C_2 >0$ only depending on $\epsilon, \gamma_2, \gamma_4$.

If $1 + \frac{r_*(\log n)^2}{n} \geq \frac{r_*\log ep (\log np)^6}{n}$, we have
\begin{align*}
\frac{\E\norm{\hat{\mathbf{\Sigma}}_0 - \E\hat{\mathbf{\Sigma}}_0}^2}{(C_1\kappa_1^2\sqrt{\frac{r_*\log ep}{n} })^2} 
\leq &1 + \frac{r_*(\log n)^2}{n}  + \int_{1 + \frac{r_*(\log n)^2}{n} }^{
	\frac{\{1 + \frac{r_*(\log n)^2}{n}\}^2}{\frac{r_*\log ep (\log np)^6}{n}}
} 2p\exp\bigg[
-\frac{ C_2 (\log ep) t/\{\log(\sqrt{\frac{r_*\log ep}{n} t})\}^2  }{1 + \frac{r_*(\log n)^2}{n}}
\bigg]\mathrm{d}t \\
&+ \int_{	\frac{\{1 + \frac{r_*(\log n)^2}{n}\}^2}{\frac{r_*\log ep (\log np)^6}{n}}}^\infty 
2p\exp\bigg[
-\frac{ C_2 (\log ep) \sqrt{t}/\{\log(\sqrt{\frac{r_*\log ep}{n} t})\}^2  }{\sqrt{\frac{r_*\log ep (\log np)^6}{n}}}
\bigg]\mathrm{d}t\\
\leq & C_3\Big( 1 + \frac{r_*(\log n)^2}{n} + \frac{r_* \log ep(\log np)^6}{n}  \Big).
\end{align*}
This gives that 
$$	\E\norm{\hat{\mathbf{\Sigma}}_0 - \E\hat{\mathbf{\Sigma}}_0}^2 \leq C_4\kappa_1^4\Big\{\frac{r_*\log ep}{n} + \frac{r_*^2(\log ep)^2(\log np)^6}{n^2}\Big\}.$$
On the other hand, if $1 + \frac{r_*(\log n)^2}{n} \leq \frac{r_*\log ep (\log np)^6}{n}$, 
\begin{align*}
\frac{\E\norm{\hat{\mathbf{\Sigma}}_0 - \E\hat{\mathbf{\Sigma}}_0}^2}{(C_1\kappa_1^2\sqrt{\frac{r_*\log ep}{n} })^2} 
\leq&  \frac{r_*\log ep(\log np)^6}{n} + \int_{ \frac{r_*\log ep(\log np)^6}{n}}^{\infty} 2p\exp\bigg[
-\frac{ C_2 (\log ep) \sqrt{t}/\{\log(\sqrt{\frac{r_*\log ep}{n} t})\}^2  }{\sqrt{\frac{r_*\log ep (\log np)^6}{n}}}
\bigg] \mathrm{d}t\\
\leq& C_5\frac{r_*\log ep(\log np)^6}{n}.
\end{align*}
This renders $$\E\norm{\hat{\mathbf{\Sigma}}_0 - \E\hat{\mathbf{\Sigma}}_0}^2 \leq C_5\kappa_1^4\Big\{\frac{r_*^2(\log ep)^2(\log np)^6}{n^2}\Big\}.$$
Combining two cases gives us the final result by using the simple fact that $\E\norm{\hat{\mathbf{\Sigma}}_0 - \E\hat{\mathbf{\Sigma}}_0} \leq	(\E\norm{\hat{\mathbf{\Sigma}}_0 - \E\hat{\mathbf{\Sigma}}_0}^2)^{\frac{1}{2}}$. This completes the proof of the first part of Theorem \ref{thm:subgaummb}.


	Notice that under Assumptions \ref{A1}, \ref{A4}, zero-mean, and second-order stationarity, we have $\kappa_1^2 \asymp \norm{\mathbf{\Sigma}_0}$ and $\kappa_*^2 \asymp \tr(\mathbf{\Sigma}_0)$. Thus plugging in the first part of Theorem \ref{thm:subgaummb} finishes the proof.
\end{proof}

Now we prove Proposition \ref{prop:tailbd} under Assumptions \ref{A1}-\ref{A3}. In the proof, the cases for covariance and autocovariance matrices are treated separately. In the following we give a roadmap. The proof of Proposition \ref{prop:tailbd} is based on combining a Bernstein-type inequality for the almost surely (a.s.) bounded matrices and a truncation method. The probability bound for the a.s. bounded part (a.k.a., the truncated part) of the random matrix is obtained by employing a Bernstein-type inequality for $\tau$-mixing random matrices, which is presented in Theorem \ref{thm:bern}, and some related lemmas (Lemmas \ref{lem:tau} and \ref{lem:nu}), whose proofs are presented later. The tail part of the random matrix is controlled under the sub-Gaussian Assumption \ref{A1}, which uses Lemma \ref{lem:subgautail} that will be presented soon.

In more detail, given a sequence of random vectors $\{\bY_t\}_{t \in \Z}$, denote $\Xb_t := \bY_t \bY_t^\T$ for all $t \in \Z$. Then for any constant $M >0$, we introduce the following ``truncated" version of $\Xb_t$: 
\[
\Xb_t^M := \frac{M \wedge \norm{\Xb_t}}{\norm{\Xb_t}}\Xb_t,
\]
where $a\wedge b:={\rm min}(a,b)$ for any two real numbers $a,b$.

For any integer $m>0$, we denote $\Zb^{(m)}_t : = \bY_t\bY_{t+m}^\T$ for all $t \in \Z$. For the sake of clarification,  the superscript ``$(m)$" is dropped when no confusion is possible. Then the truncated version is  $$\Zb^{M}_t : = \frac{M \wedge \norm{\Zb_t}}{ \norm{\Zb_t}}\Zb_t$$ for any $M>0$. 

We further define the ``variances" for $\{\Xb_i^M\}_{i = 1}^n$ and $\{\Zb_i^M\}_{i = 1}^{n-m}$  as
\begin{align*}
\nu^2_{\Xb^M} &:= \sup_{K \subseteq \{1, \dots, n\}}\frac{1}{\card{(K)}}\lambda_{\max}\bigg\{\E\bigg(\sum_{i\in K}\Xb_i^M - \E\Xb_i^M\bigg)^2\bigg\},\\
\nu^2_{\Zb^M} &:= \sup_{K \subseteq \{1, \dots, n-m\}}\frac{1}{\card{(K)}}\bigg\|\E\bigg(\sum_{i\in K}\Zb_i^M - \E\Zb_i^M\bigg)^2\bigg\|.
\end{align*}
Here $\lambda_{\max}(\Xb)$ and $\lambda_{\min}(\Xb)$ denote the largest and smallest eigenvalues of $\Xb$ respectively.

\begin{proof}[Proof of Proposition \ref{prop:tailbd}]
	We first assume $\kappa_1= 1$. We consider two cases.
	
	\textbf{Case I:} When $m = 0$, $\{\Xb_t\}_{t \in \Z}$ is a sequence of symmetric random matrices. We have,
	\begin{align}\label{eq:taildecomp}
	&\P\bigg\{\frac{1}{n}\bigg\lVert\sum_{i = 1}^{n} (\Xb_i - \E \Xb_i)\bigg\rVert \geq x \bigg\} \notag\\
	= &\P\bigg\{\frac{1}{n}\bigg\lVert\sum_{i = 1}^{n} (\Xb_i - \Xb_i^M + \Xb_i^M - \E \Xb_i^M + \E \Xb_i^M - \E \Xb_i)\bigg\rVert \geq x \bigg\} \notag\\
	\leq & \P\bigg\{\frac{1}{n}\bigg\lVert\sum_{i = 1}^{n} (\Xb_i^M - \E \Xb_i^M +\E\Xb_i^M -\E \Xb_i )\bigg\rVert +  \frac{1}{n}\bigg\lVert\sum_{i = 1}^{n} (\Xb_i -  \Xb_i^M)\bigg\rVert \geq x \bigg\} \notag\\
	\leq &\P\bigg\{\bigg\lVert\sum_{i = 1}^{n} (\Xb_i^M - \E \Xb_i^M + \E\Xb_i^M - \E \Xb_i )\bigg\rVert \geq nx\bigg\}+  \P\bigg\{\bigg\lVert\sum_{i = 1}^{n} (\Xb_i- \Xb_i^M)\bigg\rVert > 0 \bigg\} \notag\\
	\leq &\P\bigg\{\bigg\lVert\sum_{i = 1}^{n} (\Xb_i^M - \E \Xb_i^M)\bigg\rVert \geq nx-\sum_{i = 1}^n\norm{\E\Xb_i^M - \E \Xb_i} \bigg\} +  \sum_{i = 1}^{n}\P(\Xb_i \neq \Xb_i^M ) \notag\\
	\leq &\P\bigg[\lambda_{\max}\bigg\{\sum_{i = 1}^{n} (\Xb_i^M - \E \Xb_i^M)\bigg\} \geq nx-\sum_{i = 1}^n\norm{\E\Xb_i^M - \E \Xb_i} \bigg] + \notag\\
	&\P\bigg[\lambda_{\min}\bigg\{\sum_{i = 1}^{n} (\Xb_i^M - \E \Xb_i^M)\bigg\} \leq -nx+\sum_{i = 1}^n\norm{\E\Xb_i^M - \E \Xb_i} \bigg] + \sum_{i = 1}^{n}\P(\Xb_i \neq \Xb_i^M ).
	\end{align}
	
	We first show that the difference in expectation between the ``truncated" $\Xb_t^{M_\delta}$ and original one $\Xb_t$ can be controlled with the chosen truncation level $M_{\delta}$. For this, we need the following lemma.
	
	\begin{lemma}[Proof in Section \ref{subsec:aux}] \label{lem:subgautail}
		Let $\{\bY_t\}_{t\in \Z}$ be a sequence of $p$-dimensional random vectors under Assumption \ref{A1}. Then for all $t \in \Z$ and for all $x \geq 0$,
		$$\P\{\norm{\bY_t }_2^2  \geq 2\kappa_*^2 + 8\kappa_*^2(x + \sqrt{x})\}\leq \exp(-Cx)$$
		for some arbitary constant $C>0$.
	\end{lemma}
	
	By applying Lemma \ref{lem:subgautail}, we obtain that for all $i \in \{ 1, \dots, n\}$, 
		\begin{align*}
	\norm{\E\Xb_i^{M_\delta} - \E \Xb_i} 
	=&\bigg\|\E\Big(1-\frac{M_\delta}{\norm{\Xb_i}}\Big)\Xb_i\mathbf{1}_{\{\norm{\Xb_i} >M_\delta\}}\bigg\|\\
	\leq & \sup_{\bu, \bv \in \mathbb{S}^{p-1}}\E |\bu^\T \Xb_i \bv|\mathbf{1}_{\{\norm{\Xb_i} >M_\delta\}}\\
	\leq  & \sup_{\bu, \bv \in \mathbb{S}^{p-1}}\{\E (\bu^\T \bY_i\bY_i^\T \bv)^2\}^{\frac{1}{2}}\{\P(\norm{\Xb_i} >M_\delta)\}^{\frac{1}{2}}\\
	\leq &\sqrt{\delta/n},
	\end{align*}
	where the last line followed by Assumption \ref{A1}, Lemma \ref{lem:subgautail}, and the chosen $M_\delta$.
	
	The second step heavily depends on a Bernstein-type inequality for $\tau$-mixing random matrices. The theorem slightly extends the main theorem of \cite{banna2016bernstein} in which the random matrix sequence is assumed to be $\beta$-mixing. Its proof is relegated to the Appendix.
	
	\begin{theorem}[Proof in Appendix]\label{thm:bern}
		Consider a sequence of real, mean-zero, symmetric $p \times p$ random matrices $\lbrace \Xb_t \rbrace_{t\in \Z}$ with $\lVert \Xb_t \rVert \leq M$ for some positive constant $M$. In addition, assume that this sequence is $\tau$-mixing (see, Appendix Section \ref{sec:tau} for a detailed introduction to the $\tau$-mixing coefficient) with geometric decay, i.e.,
		\[
		\tau(k; \{\Xb_t\}_{t \in \Z}, \norm{\cdot}) \leq M\psi_1\exp\{-\psi_2(k-1)\} 
		\]
		for some constants $\psi_1, \psi_2 > 0$. Denote $\tilde{\psi}_1 := \max\{p^{-1}, \psi_1\}$. Then for any $ x\geq 0$ and any integer $n\geq 2$, we have 
		\begin{align*}
		\P\bigg\{\lambda_{\max}\bigg(\sum_{i = 1}^{n}\Xb_i\bigg) \geq x\bigg\}
		\leq p \exp\bigg\{-\frac{x^2}{8(15^2n\nu^2 + 60^2M^2/\psi_2) + 2xM\tilde{\psi}(\tilde{\psi}_1,\psi_2,n,p)}\bigg\},
		\end{align*}
		where 
		\begin{align*}
		&\nu^2 := \sup_{K\subseteq \lbrace 1,\dots, n \rbrace}\frac{1}{\card(K)}\lambda_{\max}\bigg\{\E\bigg(\sum_{i \in K}\Xb_i\bigg)^2\bigg\}\ \mbox{and}\ \tilde{\psi}(\tilde{\psi}_1,\psi_2,n,p) := \frac{\log n}{\log 2} \max\bigg\{1,\frac{8\log(\tilde{\psi}_1n^6p)}{\psi_2}\bigg\}.
		\end{align*}
	\end{theorem}
	
	In order to apply Theorem \ref{thm:bern}, we need the following two lemmas. Lemma \ref{lem:tau} is to show that the sequence of ``truncated" matrices $\{\Xb_t^M\}$ under Assumptions \ref{A1}-\ref{A2} is a $\tau$-mixing random sequence with geometric decay. Lemma \ref{lem:nu} calculates the upper bound for $\nu^2$ term in Theorem \ref{thm:bern} for $\{\Xb_t^M\}_{t \in \Z}$.
	
	\begin{lemma}[Proof in Section \ref{subsec:aux}] \label{lem:tau}
		Let $\{\bY_t\}_{t \in \Z}$ be a sequence of random vectors under Assumptions \ref{A1}-\ref{A2}. Then $\{ \Xb_t^M \}_{t \in \Z}$, $\{ \Xb_t^M -\E \Xb_t^M\}_{t \in \Z}$, $\{ \Zb_t^M \}_{t \in \Z}$, and $\{ \Zb_t^M -\E \Zb_t^M\}_{t \in \Z}$ are all $\tau$-mixing random sequences. Moreover, 
			\begin{align*}
		&\tau(k; \{\Xb_t^M\}_{t \in \Z}, \norm{\cdot})  \leq  C\gamma_1\kappa_1\kappa_*\exp\{-\gamma_2(k-1)\},\\
		&\tau(k; \{\Xb_t^M - \E\Xb_t^M\}_{t \in \Z}, \norm{\cdot})  \leq  C\gamma_1\kappa_1\kappa_*\exp\{-\gamma_2(k-1)\},\\
		&\tau(k; \{\Zb_t^M\}_{t \in \Z}, \norm{\cdot})  \leq  C'\exp\{\gamma_2\min(k,m)\}
		\max( \gamma_1\kappa_1\kappa_*, \kappa_*^2)\exp\{-\gamma_2(k-1)\},\\
		&\tau(k; \{\Zb_t^M - \E\Zb_t^M \}_{t \in \Z}, \norm{\cdot})  \leq  C'\exp\{\gamma_2\min(k,m)\}
		\max( \gamma_1\kappa_1\kappa_*, \kappa_*^2)\exp\{-\gamma_2(k-1)\}
		\end{align*}
		for $k \geq 1$ and some constants $C, C'>0$ only depending on $\epsilon$.
	\end{lemma}
	
	\begin{lemma}[Proof in Section \ref{subsec:aux}] \label{lem:nu}
		Let $\{\bY_t\}_{t \in \Z}$ be a sequence of random vectors under Assumptions \ref{A1}-\ref{A3}.  Take $M \geq C
		\gamma_1\kappa_1\kappa_*$ for some constant $C >0$ only depending on $\epsilon$. Then we obtain
		\begin{align*}
		&\nu^2_{\Xb^M} \leq  C'\frac{  \kappa_1^2 \{  \kappa_1^2 + \kappa_1\kappa_*\gamma_1 + \kappa_*^2(\gamma_3 + 2)\}  }{1-\exp\{-\min( \frac{5+\epsilon}{6\epsilon+10}\gamma_2, \gamma_4 )\}},\\
		&\nu^2_{\Zb^M} \leq  C''\frac{  \kappa_1^2 \{  (2m+1)\kappa_1^2 + \kappa_1\kappa_*\gamma_1 + \kappa_*^2(\gamma_3 + 2m+2)\}  }{1-\exp\{-\min( \frac{5+\epsilon}{6\epsilon+10}\gamma_2, \gamma_4 )\}}
		\end{align*}
		for some constants $C', C''>0$ only depending on $\epsilon$.
	\end{lemma}
	
	Therefore, by applying Theorems \ref{thm:bern}, Lemma \ref{lem:tau}, and Lemma \ref{lem:nu} with the chosen $M_\delta$, we obtain
	for any $x > 0$,
	\begin{align}\label{lambdamax}
	&\P\bigg[\lambda_{\max}\bigg\{\frac{1}{n}\sum_{i = 1}^{n} (\Xb_i^{M_\delta} - \E \Xb_i^{M_\delta})\bigg\} \geq x +\sqrt{\delta/n}\bigg] \leq p\exp\bigg(-\frac{n^2x^2}{A_1n + A_2M_\delta^2 + A_3nxM_{\delta}} \bigg),
	\end{align}
	where 
	\begin{align*}
	&A_1 := \frac{ C\{
		\kappa_*\gamma_1 + \kappa_*^2(\gamma_3 + 2) + 1
		\}}{1-\exp\{-\min( \frac{5+\epsilon}{6\epsilon+10}\gamma_2, \gamma_4 )\}},\ \ A_2 := \frac{453^2}{\gamma_2},\ \mbox{and}\ A_3 := \frac{2\log n}{\log 2}\max\bigg\{1, \frac{48\log (np)}{\gamma_2} \bigg\}
	\end{align*}
	for some constant $C>0$ only depending on $\epsilon$.
	
	Similarly, notice that $\lambda_{\min}(\sum_{j = 1}^n \Xb_j^{M_\delta}) = \lambda_{\max}(-\sum_{j = 1}^n \Xb_j^{M_\delta})$. Hence the same argument renders the same upper bound
	\begin{align}\label{lambdamin}
	&\P\bigg[\lambda_{\min}\bigg\{\frac{1}{n}\sum_{i = 1}^{n} (\Xb_i^{M_\delta} - \E \Xb_i^{M_\delta})\bigg\} \leq -(x +\sqrt{\delta/n})\bigg]
	\leq p\exp\bigg( -\frac{n^2x^2}{A_1n + A_2M_\delta^2 + A_3nxM_{\delta}} \bigg)
	\end{align}
	with the same constants as above.
	
	For the last term of (\ref{eq:taildecomp}), with the choice of $M_\delta$ and Lemma \ref{lem:subgautail}, we obtain
	\begin{align}\label{sumtail}
	\sum_{i = 1}^{n}\P(\Xb_i \neq \Xb_i^{M_\delta})  = \sum_{i = 1}^{n}\P(\norm{\Xb_i} > M_\delta) \leq \delta.
	\end{align} 
	
	Combining (\ref{lambdamax}), (\ref{lambdamin}), and (\ref{sumtail}), we obtain
	\begin{align*}
	&	\P(\norm{\hat{\mathbf{\Sigma}}_0 - \E\hat{\mathbf{\Sigma}}_0} \geq x +\sqrt{\delta/n})
	\leq 2p\exp\bigg(-\frac{n^2x^2}{A_1n + A_2M_\delta^2 + A_3nxM_{\delta}} \bigg)+ \delta
	\end{align*}
	with the constants $A_1, A_2, A_3$ defined above.
	
	\textbf{Case II:} Now we consider the case when $0<m \leq n-2$. Since $\Zb_t := \bY_t\bY_{t + m}^\T$ is not symmetric for all $t \in\Z$, by applying matrix dilation (See \cite{tropp2015introduction}, Section 2.1.16 for more details), we define the symmetric version of $\Zb_t^M$ as
	$$\overline{\Zb}_t^M :=\begin{bmatrix}
	\mathbf{0} &\Zb_t^M \\
	(\Zb_t^M)^\T& \mathbf{0}\\
	\end{bmatrix}.$$
	Observe that $\lambda_{\max}(\overline{\Zb}_t^M ) = \norm{\overline{\Zb}_t^M} = \norm{\Zb_t^M}$.
	By Lemma \ref{lem:tau}, $\{\overline{\Zb}_t^M\}_{t \in \Z}$ and  $\{\overline{\Zb}_t^M - \E\overline{\Zb}_t^M \}_{t \in \Z}$ are also sequences of $\tau$-mixing random matrices. Define $$\nu^2_{\overline{\Zb}^M } := \sup_{K \subseteq \{1, \dots, n-m\}}\frac{1}{\card{(K)}}\lambda_{\max}\bigg\{\E\bigg(\sum_{i\in K}\overline{\Zb}_i^M  - \E\overline{\Zb}_i^M \bigg)^2\bigg\}.$$
	Notice that $\nu^2_{\overline{\Zb}^M}$ and $\nu^2_{\Zb^M}$ have the same upper bound since spectral norm of block diagonal matrix is less than or equal to the spectral norm of each block.
	
	Now we apply similar arguments in Case I to $\{\overline{\Zb}_t\}_{t\in \Z}$ and $\{\overline{\Zb}_t^M\}_{t \in \Z}$.
	\begin{align*}
	&\P\bigg\{\frac{1}{n-m}\bigg\lVert\sum_{i = 1}^{n-m} (\Zb_i - \E \Zb_i)\bigg\rVert \geq x \bigg\} \\
	\leq &\P\bigg[\lambda_{\max}\bigg\{\sum_{i = 1}^{n-m} (\overline{\Zb}_i^M - \E \overline{\Zb}_i^M)\bigg\} \geq (n-m)x-\sum_{i = 1}^{n-m}\norm{\E\overline{\Zb}_i - \E \overline{\Zb}_i^M} \bigg] +  \sum_{i = 1}^{n-m}\P\Big(\Zb_i \neq \Zb_i^{M} \Big).
	\end{align*}
	The rest is straightforward by using Theorem \ref{thm:bern}, Lemma \ref{lem:subgautail}, Lemma \ref{lem:tau}, and Lemma \ref{lem:nu}, and we thus finish the rest of the proof. \\
	
	Lastly, we consider $\kappa_1\neq1$. Notice that for any sequence $\{\bY_t\}_{t \in \Z}$ satisfying Assumptions $\ref{A1}$-$\ref{A3}$, the sequence $\{\bY_t/\kappa_1\}_{t \in \Z}$ will satisfy Assumptions \ref{A1} automatically and Assumptions \ref{A2}-\ref{A3} with $\kappa_1=1$. Hence, applying the above to $\{\bY_t/\kappa_1\}_{t \in \Z}$ renders the results. This completes the proof of Proposition \ref{prop:tailbd}.
\end{proof}

\subsection{Proof of Theorem \ref{thm: gaummb}}\label{subsec:thm2.2}
\begin{proof}
The proof of Theorem \ref{thm: gaummb} consists of two cases.

\textbf{Case I.} When $m = 0$, we first state a more general result of Gaussian process. Proposition \ref{prop:gau} considers a general Gaussian process without further assumptions on the covariance and autocovariance matrices. The proof modifies  that of Theorem 5.1 in \cite{van2017structured} with dependence among observations taken into account.

\begin{proposition}[Proof in Section \ref{subsec:thm2.2}] \label{prop:gau}
	Let $\{\bY_t\}_{t \in \Z}$ be a stationary sequence of mean-zero Gaussian random vectors with autocovariance matrices $\mathbf{\Sigma}_m$ for $0 \leq m \leq n-1$. Then 
\begin{align*}
\E\norm{\hat{\mathbf{\Sigma}}_0- \mathbf{\Sigma}_0} \leq &\frac{2}{n}\bigg\{2 \Big(\norm{\mathbf{\Sigma}_0}_* + 2\sum_{m = 1}^{n-1} \norm{\mathbf{\Sigma}_m}_*\Big) +\sqrt{2n\norm{\mathbf{\Sigma}_0}\Big(\norm{\mathbf{\Sigma}_0}_* + 2\sum_{m = 1}^{n-1} \norm{\mathbf{\Sigma}_m}_*\Big)} \\
&+\sqrt{2n\Big(\norm{\mathbf{\Sigma}_0} + 2\sum_{m = 1}^{n-1} \norm{\mathbf{\Sigma}_m}\Big)\tr(\mathbf{\Sigma}_0)}\bigg\},
\end{align*}
where $\norm{\cdot}_*$ is the matrix nuclear norm.
\end{proposition}

The rest of the proof is to show the geometric decay of spectral norm and nuclear norm of autocovariance matrices under Assumptions \ref{A2}-\ref{A3} in order to apply Proposition \ref{prop:gau}. It is obvious that $\kappa_1^2 \asymp \norm{\mathbf{\Sigma}_0}$ and $\kappa_*^2 \asymp \tr(\mathbf{\Sigma}_0)$ when the process is a centered stationary Gaussian process. We first prove the geometric decay of spectral norm of autocovariance matrices. For any $0 \leq m \leq n-1$ and any integer $j$, by Assumption \ref{A3}, there exists $\tilde{\bY}_{1+m}$ that is identically distributed as $\bY_{1+m}$, independent of $\bY_{1}$, and 
$$\sup_{\bu \in \mathbb{S}^{p-1}} \norm{(\bY_{1+m}  - \tilde{\bY}_{1+m})^\T \bu}_{L(1+\epsilon)} \leq \gamma_3\sqrt{\norm{\mathbf{\Sigma}_0}}\exp\{-\gamma_4(m-1)\}.$$
Therefore,
\begin{align*}
\norm{\mathbf{\Sigma}_m} = &\norm{\E\bY_1\bY_{1+m}^\T} \\
=& \norm{\E\bY_1(\bY_{1+m} - \tilde{\bY}_{1+m} + \tilde{\bY}_{1+m})^\T} \\
=& \norm{\E\bY_1(\bY_{1+m} - \tilde{\bY}_{1+m})^\T}\\
\leq & \sup_{\bu, \bv \in \mathbb{S}^{p-1}} |\E \bu^\T\bY_1(\bY_{1+m} - \tilde{\bY}_{1+m})^\T \bv |\\
\leq & C\norm{\mathbf{\Sigma}_0}\exp\{-\gamma_4(m-1)\},
\end{align*}
where the last inequality is followed by Assumption \ref{A3} and $\gamma_3 = O(1)$ for some constant $C>0$ only depending on $\epsilon, \gamma_3$.

Similarly, by Assumption \ref{A2}, there exists $\tilde{\bY}_{1+m}$ that is identically distributed as $\bY_{1+m}$, independent of $\bY_{1}$, and 
$$\norm{\norm{\bY_{1+m}  - \tilde{\bY}_{1+m}}_2}_{L(1+\epsilon)} \leq \gamma_1\sqrt{\norm{\mathbf{\Sigma}_0}}\exp\{-\gamma_2(m-1)\}.$$
Then,
\begin{align*}
\norm{\mathbf{\Sigma}_m}_*  &= \sqrt{\tr(\mathbf{\Sigma}_m^\T \mathbf{\Sigma}_m)}\\
& = \sqrt{\tr\{\E(\bY_{1+m} - \tilde{\bY}_{1+m})\bY_1^\T \E\bY_1(\bY_{1+m} - \tilde{\bY}_{1+m})^\T   \}}\\
&\leq \sqrt{ \tr \{\E (\bY_{1+m} - \tilde{\bY}_{1+m})\bY_1^\T \bY_1(\bY_{1+m} - \tilde{\bY}_{1+m})^\T   \}   }\\
&= \sqrt{  \tr \{\E \bY_1^\T \bY_1(\bY_{1+m} - \tilde{\bY}_{1+m})(\bY_{1+m} - \tilde{\bY}_{1+m})^\T   \}  }\\
& = \sqrt{ \E \norm{\bY_1}_2^2\norm{\bY_{1+m} - \tilde{\bY}_{1+m}}_2^2  }\\
&\leq \norm{\norm{\bY_1}_2}_{L(\frac{1+\epsilon}{\epsilon})}\norm{\norm{\bY_{1+m} - \tilde{\bY}_{1+m}}_2}_{L(1+\epsilon)}\\
&\leq C \tr(\mathbf{\Sigma}_0)\exp\{-\gamma_2(m-1)\},
\end{align*}
where the third line is followed by the fact that $\E(\bY_{1+m} - \tilde{\bY}_{1+m})\bY_1^\T \E\bY_1(\bY_{1+m} - \tilde{\bY}_{1+m})^\T \preceq \E \bY_1^\T \bY_1(\bY_{1+m} - \tilde{\bY}_{1+m})(\bY_{1+m} - \tilde{\bY}_{1+m})^\T$ (``$\preceq$" is the Loewner partial order of Hermitian matrices), and both matrices are positive semi-definite, and the last line by Assumption \ref{A2} and $\gamma_1 = O(\sqrt{r(\bSigma_0)})$. Indeed, for any $\bu \in \R^p$, $\E\{\bu^\T (\bY_{1+m} - \tilde{\bY}_{1+m})\}^2(\bY_1^\T\bY_1) = \sum_{j = 1}^p \E \{\bu^\T (\bY_{1+m} - \tilde{\bY}_{1+m})\}^2\bY_{1,j}^2$ and $\E\{ \bu^\T(\bY_{1+m} - \tilde{\bY}_{1+m}) \}\bY_1^\T  \E\bY_1(\bY_{1+m} - \tilde{\bY}_{1+m})^\T\bu = \sum_{j = 1}^p [\E\{\bu^\T(\bY_{1+m} - \tilde{\bY}_{1+m}) \bY_{1,j} \} ]^2$. The result follows.

\textbf{Case II.}  When $m>0$, we denote $\bar{\bY}_i:= (\bY_i^\T\ \bY_{i+m}^\T)^\T$ for $1\leq i \leq n-m$. It is obvious that $\{\bar{\bY}_i\}$ is a centered stationary Gaussian process satisfying Assumptions \ref{A2}-\ref{A3}. Denote $\bar{\mathbf{\Sigma}}_0 := \E \bar{\bY}_i\bar{\bY}_i^\T$ and notice that $\mathbf{\Sigma}_m$ is the off-diagnal block submatrix of $\bar{\mathbf{\Sigma}}_0$. By Case I and the fact that spectral norm of submatrix is bounded above by that of the full matrix, we obtain
$$ \E\norm{\hat{\mathbf{\Sigma}}_m- \mathbf{\Sigma}_m} \leq C\norm{\bar{\mathbf{\Sigma}}_0}\Big(\sqrt{\frac{r(\bSigma_0)}{n-m}} + \frac{r(\bSigma_0)}{n-m} \Big).$$
Notice that $\norm{\bSigma_0}\leq \norm{\bar{\mathbf{\Sigma}}_0} \leq \norm{\mathbf{\Sigma}_0} + \norm{\mathbf{\Sigma}_m} \leq 2\norm{\mathbf{\Sigma}_0}$ since $\mathbf{\Sigma}_0-\mathbf{\Sigma}_m$ is positive semi-definite. This completes the proof.
\end{proof}

\begin{proof}[Proof of Proposition \ref{prop:gau}]
The proof heavily depends on the following observation. Denote $\Yb := (\bY_1 \dots \bY_n)$ and let $\tilde{\Yb}$ be an independent copy of $\Yb$. Then
$$\E\norm{\hat{\mathbf{\Sigma}}_0- \mathbf{\Sigma}_0} \leq \frac{2}{n} \E\norm{\Yb \tilde{\Yb}^{\T}}.$$
This is same as Lemma 5.2 in \cite{van2017structured} by noticing that the result holds without independence assumption. 

Now we state the following two core lemmas used to complete the proof.

\begin{lemma}[Proof in Section \ref{subsec:aux}] \label{lem:gau1} We have
		\begin{align*}
	\E\norm{\hat{\mathbf{\Sigma}}_0- \mathbf{\Sigma}_0} \leq \frac{2\sqrt{2}}{n}\bigg\{\E\norm{\Yb}  \cdot \sqrt{\tr\Big(\mathbf{\Sigma}_0 + 2\sum_{d = 1}^{n-1}\tilde{\mathbf{\Sigma}}_d\Big)} +\sqrt{2\Big(\norm{\mathbf{\Sigma}_0} + 2\sum_{d = 1}^{n-1} \norm{\mathbf{\Sigma}_d}\Big)}\cdot \sqrt{n\tr(\mathbf{\Sigma}_0)}\bigg\},
	\end{align*}
	where $\tilde{\mathbf{\Sigma}}_d := (\Ub_d \mathbf{\Lambda}_d \Ub_d^{\T} +  \Vb_d \mathbf{\Lambda}_d \Vb_d^{\T})/2$. Here $\Ub_d,  \Vb_d, \mathbf{\Lambda}_d$ are left singular vectors, right singular vectors, and singular values of $\mathbf{\Sigma}_d$ for all $1 \leq d\leq n-1$ respectively.
\end{lemma}

\begin{lemma}[Proof in Section \ref{subsec:aux}] \label{lem:gau2} We have
	\begin{align*} 
	\E\norm{\Yb}  \leq  \sqrt{2\tr\bigg(\mathbf{\Sigma}_0 + 2\sum_{d = 1}^{n-1}\tilde{\mathbf{\Sigma}}_d\bigg)} + \sqrt{2n\norm{\mathbf{\Sigma}_0} },
	\end{align*}
	where $\tilde{\mathbf{\Sigma}}_d$ for all $1 \leq d \leq n-1$ are defined in Lemma \ref{lem:gau1}.
\end{lemma}

The proof of Proposition \ref{prop:gau} completes by combining Lemma \ref{lem:gau1} and Lemma \ref{lem:gau2}.
\end{proof}

\subsection{Proofs of auxiliary lemmas}\label{subsec:aux}

\begin{proof}[Proof of Lemma \ref{lem:subgautail}]
	By Lemma A.2 in \cite{bunea2015sample}, we have $\E\norm{\bY_t}_2^{2k} \leq (2k)^k\kappa_*^{2k}$ for $t \in \Z$. Hence $$\norm{\norm{\bY_t}_2^2 - \E\norm{\bY_t}_2^2}_{\psi_1} \leq  2\norm{\norm{\bY_t}_2^2}_{\psi_1} \leq 4\norm{\norm{\bY_t}_2}_{\psi_2}^2\leq 8\kappa_*^2.$$
	
	Thus by property of sub-exponential random variable and Chernoff inequality, we have for any $x \geq 0$,
	$$\P(\norm{\bY_t}_2^2 - \E\norm{\bY_t}_2^2 \geq x) \leq \exp\Big\{-C\min\Big(\ \frac{x^2}{64\kappa_*^4}, \frac{x}{8\kappa_*^2} \Big)\Big\},$$
	for some arbitary constant $C>0$.
	Obviously, we have for all $x \geq 0$,
	$$\P\{\norm{\bY_t }_2^2  \geq 2\kappa_*^2 + 8\kappa_*^2(x + \sqrt{x})\}\leq \exp(-Cx)$$
	for some arbitary constant $C>0$. This completes the proof.
\end{proof}


\begin{proof}[Proof of Lemma \ref{lem:tau}]
	We first show that $\{\Xb_t\}_{t\in \Z}$ is a sequence of $\tau$-mixing random vectors with geometric decay. Under Assumption \ref{A2} (without loss of generality, take $j=0$), there exists a sequence of random vectors $\{\tilde{\bY}_t\}_{t > 0}$ which is independent of $\sigma(\{\bY_t\}_{t \leq 0})$, identically distributed as $\{\bY_t\}_{t > 0}$, and for any integer $t \geq 1$,
	$$\norm{\norm{\bY_t- \tilde{\bY}_t}_2}_{L(1+\epsilon)} \leq \gamma_1 \kappa_1  \exp\{-\gamma_2(t-1)\}$$
	for some constant $\epsilon >0$.
	Then for any $m \geq 0$,
	\begin{align*}
	&\E\norm{\bY_t\bY_{t+m}^\T- \tilde{\bY}_t\tilde{\bY}_{t+m}^\T} \\
	= &\E\norm{\bY_t\bY_{t+m}^\T- \bY_t\tilde{\bY}_{t+m}^\T + \bY_t\tilde{\bY}_{t+m}^\T- \tilde{\bY}_t\tilde{\bY}_{t+m}^\T} \\
	\leq &\E\norm{\bY_t(\bY_{t+m}- \tilde{\bY}_{t+m})^\T} + \E\norm{(\bY_t- \tilde{\bY}_t)\tilde{\bY}_{t+m}^\T}\\
	\leq & \norm{\norm{\bY_t}_2}_{L(\frac{1+\epsilon}{\epsilon})}
	\norm{\norm{\bY_{t+m}- \tilde{\bY}_{t+m}}_2}_{L(1+\epsilon)}
	+  \norm{\norm{\bY_{t+m}}_2}_{L(\frac{1+\epsilon}{\epsilon})}
	\norm{\norm{\bY_{t}- \tilde{\bY}_{t}}_2}_{L(1+\epsilon)}\\
	\leq & C \gamma_1 \kappa_1\kappa_* \exp\{-\gamma_2(t-1)\},
	\end{align*}	
	where the fourth line is followed by H\"older's inequality and the fact that
	$$\sup_{t \in \Z} \norm{\norm{\bY_t}_2 }_{L(\alpha)} \leq  \sup_{t \in \Z} \sup_{\bu\in \bar{\mathbb{S}}^{p-1}} \norm{\bu^\T \bY_t}_{L(\alpha)} \leq  \sup_{t \in \Z} \sup_{\bu\in \bar{\mathbb{S}}^{p-1}} \sqrt{\alpha} \|\bu^\T \bY_t \|_{\psi_2} \leq \sqrt{\alpha} \kappa_*$$
	for any $\alpha \geq 1$. Here $C>0$ is some constant only depending on $\epsilon$.
	
	Now define $\tilde{\Xb}_t := \tilde{\bY}_t\tilde{\bY}_t^\T$ for any integer $t >0$. It is obvious that $\{\tilde{\Xb}_t\}_{t > 0}$ is independent of $\{\Xb_t\}_{t \leq 0}$ and identically distributed as $\{\Xb_t\}_{t > 0}$. By applying Lemma \ref{lem:tauup}, for any indices $0<k \leq t_1 < \dots < t_\ell$, we obtain
	\begin{align*}
	\tau\{\sigma(\{\Xb_t\}_{t\leq 0}), (\Xb_{t_1}, \dots, \Xb_{t_\ell}); \norm{\cdot}\}
	\leq \sum_{i= 1}^\ell \E\norm{\Xb_{t_i} - \tilde{\Xb}_{t_i}}
	\leq C\gamma_1\kappa_1\kappa_*\ell\exp\{-\gamma_2(k-1)\}.
	\end{align*}
	By definition of $\tau$-mixing coefficient, this yields
	$$\tau(k; \{\Xb_t\}_{t \in \Z}, \norm{\cdot})  \leq C\gamma_1\kappa_1\kappa_*\exp\{-\gamma_2(k-1)\}$$
	for some constant $C>0$ only depending on $\epsilon$.
	
	Now we proceed to prove $\tau$-mixing properties for the ``truncated version". The following lemma is needed. 
	
	\begin{lemma}[Proof in Section \ref{subsec:aux}] \label{lem:trun}
		Let $\bu_1,\ \bu_2,\ \bv_1,\ \bv_2 \in \R^p$ for $p\geq 1$ with unit length under $\ell_2$-norm and $\sigma_u \geq 0$. Then the function
		$$f(\sigma_v) = \norm{\sigma_v \bv_1 \bv_2^\T-\sigma_u \bu_1 \bu_2^\T}$$
		is non-decreasing in the range $\sigma_v \in [\sigma_u, \infty]$. In particular, for any $M \geq 0$ such that $M \leq \sigma_u,\ M \leq \sigma_v$, we have $$\norm{M \bv_1 \bv_2^\T - M \bu_1 \bu_2^\T} \leq \norm{\sigma_v \bv_1 \bv_2^\T - \sigma_u \bu_1 \bu_2^\T}.$$
	\end{lemma}
	
		Now consider three cases. 
	
	(1) When $\norm{\Xb_t} \leq M$ and $\norm{\tilde{\Xb}_t} \leq M$, $\norm{\Xb_t^M - \tilde{\Xb}_t^M} = \norm{\Xb_t - \tilde{\Xb}_t}$. 
	
	(2) When $\norm{\Xb_t} \leq M$ and $\norm{\tilde{\Xb}_t} > M$, we have 
	\[
	\Xb_t^M = \Xb_t = \norm{\bY_t}_2^2\frac{\bY_t}{\norm{\bY_t}_2}\frac{\bY_t^\T}{\norm{\bY_t}_2}~~~{\rm and}~~~\tilde{\Xb}_t^M = M\frac{\tilde{\bY}_t}{\norm{\tilde{\bY}_t}_2}\frac{\tilde{\bY}_t^\T}{\norm{\tilde{\bY}_t}_2}. 
	\]
	Since $\frac{\bY_t}{\norm{\bY_t}_2},\ \frac{\tilde{\bY}_t}{\norm{\tilde{\bY}_t}_2}$ have unit length and $\norm{\bY_t}_2^2\leq M < \norm{\tilde\bY_t}_2^2$, we have $\norm{\Xb_t^M - \tilde{\Xb}_t^M} \leq \norm{\Xb_t - \tilde{\Xb}_t}$ by Lemma \ref{lem:trun}. By symmetry, the same argument also applies to the case where $\norm{\Xb_t} > M$ and $\norm{\tilde{\Xb}_t} \leq  M$. 
	
	(3) When $\norm{\Xb_t} > M$ and $\norm{\tilde{\Xb}_t} > M$, we have $\Xb_t^M = M\frac{\bY_t}{\norm{\bY_t}_2}\frac{\bY_t^\T}{\norm{\bY_t}_2}$ and $\tilde{\Xb}_t^M = M\frac{\tilde{\bY}_t}{\norm{\tilde{\bY}_t}_2}\frac{\tilde{\bY}_t^\T}{\norm{\tilde{\bY}_t}_2}$. Again by Lemma \ref{lem:trun}, we have $\norm{\Xb_t^M - \tilde{\Xb}_t^M} \leq \norm{\Xb_t- \tilde{\Xb}_t}$. 
	
	By combining three cases, $\norm{\Xb_t^M - \tilde{\Xb}_t^M} \leq \norm{\Xb_t - \tilde{\Xb}_t}$ always holds, and hence $\E\norm{\Xb_t^M - \tilde{\Xb}_t^M} \leq \E\norm{\Xb_t - \tilde{\Xb}_t}$ for any $t \geq 1$. Hence for any indices $0<k \leq t_1 < \dots < t_\ell$, by Lemma \ref{lem:tauup}, we have
	$$\tau\{\sigma(\{\Xb_t^M\}_{t\leq 0}), (\Xb_{t_1}^M, \dots, \Xb_{t_\ell}^M); \norm{\cdot}\}
	\leq C\gamma_1\kappa_1\kappa_*\ell \exp\{-\gamma_2(k-1)\}$$
	for some constant $C>0$ only depending on $\epsilon$. By definition of $\tau$-mixing coefficient, this yields
	$$\tau(k; \{\Xb_t^M\}_{t \in \Z}, \norm{\cdot})  \leq C\gamma_1\kappa_1\kappa_*\exp\{-\gamma_2(k-1)\}$$
	for some constant $C>0$ only depending on $\epsilon$. Notice that $ \E\norm{\Xb_t^M - \E\Xb_t^M - (\tilde{\Xb}_t^M - \E\tilde{\Xb}_t^M)} =  \E\norm{\Xb_t^M -\tilde{\Xb}_t^M}$ since $\E\tilde{\Xb}_t^M = \E\Xb_t^M$ for any $t\geq 1$. The $\tau$-mixing property stated above applies to $\{\Xb_t^M - \E\Xb_t^M \}$ directly.

	Similar arguments apply to $\{\Zb_t^M\}_{t \in \Z}$ and $\{\Zb_t^M - \E\Zb_t^M\}_{t \in \Z}$ so we omit the details. This completes the proof.
\end{proof}


\begin{proof}[Proof of Lemma \ref{lem:nu}]
	The proof consists of two steps.
	
\textbf{Step I. } We first provide an upper bound for $\nu^2_{\Xb}$. 
	Without loss of generality, 
	we only consider $\norm{\E(\Xb_0 - \E\Xb_0)(\Xb_k - \E\Xb_k)}$ for $k \geq 0$. Under Assumptions \ref{A2}-\ref{A3}, there exists 
	$\tilde{\bY}_k$ where $\tilde{\bY}_k$ is independent of $\sigma(\{\bY_t\}_{t \leq 0 })$, identically distributed as $\bY_k$, and 
	\begin{align*}
	\norm{ \norm{\bY_k- \tilde{\bY}_k}_2  }_{L(1+\epsilon)} \leq \gamma_1\kappa_1\exp\{-\gamma_2(k-1)\},\\
	\norm{ (\bY_k- \tilde{\bY}_k)^\T\bu  }_{L(1+\epsilon)} \leq \gamma_3\kappa_1\exp\{-\gamma_4(k-1)\}
	\end{align*}
	for constants $\gamma_1, \gamma_2, \gamma_3, \gamma_4>0$ in Assumptions \ref{A2}-\ref{A3}.
	
	For $k = 0$, we have  
	$$\norm{\E \Xb_0  \Xb_0 - \E\Xb_0 \E\Xb_0} \leq C(\kappa_1^4 + \kappa_1^2\kappa_*^2)$$
	by Assumption \ref{A1} for some universal constant $C>0$. For $k >0$, we obtain
	\begin{align*}
	\norm{\E \Xb_0  \Xb_k - \E\Xb_0 \E\Xb_k}
	=& \norm{\E \Xb_0  \Xb_k -\E \Xb_0  \tilde{\Xb}_k}\\
	=&\norm{\E \Xb_0(\Xb_k-\tilde{\Xb}_k)}\\
	=& \sup_{\bu, \bv \in \mathbb{S}^{p-1}}\E|\bu^\T \bY_0\bY_0^\T(\bY_k\bY_k^\T- \tilde{\bY}_k\tilde{\bY}_k^\T)\bv|\\
	\leq &  \sup_{\bu, \bv \in \mathbb{S}^{p-1}} \E|\bu^\T \bY_0\bY_0^\T\bY_k(\bY_k^\T - \tilde{\bY}_k^\T)\bv + \bu^\T \bY_0\bY_0^\T(\bY_k- \tilde{\bY}_k)\tilde{\bY}_k^\T\bv|\\
	\leq &  \sup_{\bu, \bv \in \mathbb{S}^{p-1}} \{\E|\bY_0^\T\bY_k|^{\frac{3(1+\epsilon)}{2\epsilon}}\}^{\frac{2\epsilon}{3(1+\epsilon)}}
	\norm{\bu^\T\bY_0}_{L(\frac{3(1+\epsilon)}{\epsilon})}
	\norm{(\bY_k - \tilde{\bY}_k)^\T\bv}_{L(1+\epsilon)}+\\
	& \{\E|\bu^\T\bY_0\tilde{\bY}_k^\T\bv|^{\frac{3(1+\epsilon)}{2\epsilon}}\}^{\frac{2\epsilon}{3(1+\epsilon)}}
	\norm{\norm{\bY_0}_2}_{L(\frac{3(1+\epsilon)}{\epsilon})}
	\norm{\norm{\bY_k - \tilde{\bY}_k}_2}_{L(1+\epsilon)}\\
	\leq & C\kappa_1^2\kappa_*(\kappa_*\gamma_3 + \kappa_1\gamma_1)\exp\{-\min(\gamma_2, \gamma_4)(k-1)\},
	\end{align*}
	where the first line is followed by $\E\Xb_k =  \E\tilde{\Xb}_k$, fifth line by H\"older's inequality, and sixth line by Assumptions \ref{A1}-\ref{A3} for some constant $C>0$ only depending on $\epsilon$.
	
	Hence for any $K \subseteq \{1,\dots, n\}$,
	\begin{align*}
	&\frac{1}{\card{(K)}}\lambda_{\max}\bigg\{\E\bigg(\sum_{i \in K} \Xb_i - \E \Xb_i\bigg)^2\bigg\} \\
	\leq &\frac{1}{\card{(K)}}\bigg\|\sum_{i,j \in K} \E(\Xb_i - \E \Xb_i)(\Xb_j - \E \Xb_j)\bigg\| \\
	\leq &\frac{1}{\card{(K)}}\sum_{i,j \in K}\norm{\E(\Xb_i - \E \Xb_i)(\Xb_j - \E \Xb_j)}\\
	\leq &C\Big[ \kappa_1^4 + \kappa_1^2\kappa_*^2 + \frac{\kappa_1^2\kappa_*(\kappa_*\gamma_3 + \kappa_1\gamma_1)}{\card{(K)}}\sum_{i,j \in K, i\neq j}\exp\{-\min(\gamma_2, \gamma_4)(|i-j|-1)\}\Big]\\
	\leq & C\Big[\frac{\kappa_1^2\{\kappa_1^2 + \kappa_1\kappa_*\gamma_1  + \kappa_*^2(\gamma_3+1)\}}{1- \exp(-\min\{\gamma_2, \gamma_4\})}\Big].
	\end{align*}

\textbf{Step II.}We first bound $\nu^2_{\Xb^M}$. By definition, we have
\begin{align*}
&\bigg\|\E\bigg(\sum_{i\in K}\Xb_i^M - \E\Xb_i^M\bigg)^2\bigg\| 
=  \bigg\|\sum_{i, j \in K}\E (\Xb_i^M - \E\Xb_i^M)(\Xb_j^M - \E\Xb_j^M)\bigg\|
= \bigg\|\sum_{i, j \in K} (\E\Xb_i^M \Xb_j^M -   \E\Xb_i^M \E\Xb_j^M)\bigg\|.
\end{align*}
Without loss of generality, we consider $\norm{\E\Xb_0^M\Xb_k^M - \E\Xb_0^M \E\Xb_k^M}$ for $k\geq0$. Let $\tilde{\Xb}_k^M$ be defined as in the proof of Lemma \ref{lem:tau}. Then $\tilde{\Xb}_k^M$ is independent of 
$\tilde{\Xb}_0^M$ and distributed as $\Xb_k^M$. Hence
\begin{align*}
\norm{\E\Xb_0^M\Xb_k^M - \E\Xb_0^M \E\Xb_k^M} 
= \norm{\E\Xb_0^M\Xb_k^M  - \E\Xb_0^M \E\tilde{\Xb}_k^M}.
\end{align*}
Then we could rewrite
\begin{align*}
\norm{\E\Xb_0^M\Xb_k^M - \E\Xb_0^M \E\tilde{\Xb}_k^M}
= &\norm{\E \Xb_0\Xb_k\zeta_0\zeta_k- \E \Xb_0\tilde{\Xb}_k\zeta_0\tilde{\zeta}_k}\\
=&\norm{\E \Xb_0(\Xb_k-\tilde{\Xb}_k)\zeta_0\zeta_k + \E\Xb_0\tilde{\Xb}_k\zeta_0(\zeta_k-\tilde{\zeta}_k)},
\end{align*}
where $\zeta_i = \frac{M \wedge \norm{\Xb_i}}{\norm{\Xb_i}},\ \tilde{\zeta}_i = \frac{M \wedge \norm{\tilde{\Xb}_i}}{\norm{\tilde{\Xb}_i}}$.
Since $\zeta_0,\ \zeta_k$ are bounded by 1, we have
\begin{align*}
\norm{\E \Xb_0(\Xb_k-\tilde{\Xb}_k)\zeta_0\zeta_k} &=\sup_{\bu,\bv\in\mathbb{S}^{p-1}}\E|\bu^\T\Xb_0(\Xb_k-\tilde{\Xb}_k)\zeta_0\zeta_k \bv|\\
&\leq\sup_{\bu,\bv\in\mathbb{S}^{p-1}}\E|\bu^\T\Xb_0(\Xb_k-\tilde{\Xb}_k) \bv| \\
&=  \norm{\E \Xb_0(\Xb_k-\tilde{\Xb}_k)} \\
&\leq C\kappa_1^2(\kappa_1\kappa_*\gamma_1 + \kappa_*^2\gamma_3)
\exp\{-\min(\gamma_2, \gamma_4)(k-1)\},
\end{align*}
where the last inequality is from result in Step I for some constant $C>0$ only depending on $\epsilon$.

On the other hand, by applying H\"older's inequality, we have
\begin{align*}
\norm{\E\Xb_0\tilde{\Xb}_k\zeta_0(\zeta_k-\tilde{\zeta}_k)}
&=\sup_{\bu, \bv\in \mathbb{S}^{p-1}}\E|\bu^\T\Xb_0\tilde{\Xb}_k\bv||\zeta_k - \tilde{\zeta}_k|\\
& \leq \sup_{\bu, \bv\in \mathbb{S}^{p-1}}\{\E|\bu^\T\bY_0\bY_0^\T\tilde{\bY}_k\tilde{\bY}_k^\T\bv|^{\frac{5(1+\epsilon)}{4\epsilon}}\}^{\frac{4\epsilon}{5(1+\epsilon)}}\{\E|\zeta_k - \tilde{\zeta}_k|^{\frac{5(1+\epsilon)}{5+\epsilon}}\}^{\frac{5+\epsilon}{5(1+\epsilon)}}.
\end{align*}

Hence, for any $\bu, \bv\in \mathbb{S}^{p-1}$,
\begin{align*}
\{\E|\bu^\T\bY_0\bY_0^\T\tilde{\bY}_k\tilde{\bY}_k^\T\bv|^{\frac{5(1+\epsilon)}{4\epsilon}}\}^{\frac{4\epsilon}{5(1+\epsilon)}} \leq& \norm{\bu^{\T}\bY_0}_{L(\frac{5(1+\epsilon)}{\epsilon})}\norm{\bu^{\T}\tilde{\bY}_k}_{L(\frac{5(1+\epsilon)}{\epsilon})} \norm{\norm{\bY_0}_2}_{L(\frac{5(1+\epsilon)}{\epsilon})}\norm{\norm{\tilde{\bY}_k}_2}_{L(\frac{5(1+\epsilon)}{\epsilon})}\\
\leq &C\kappa_1^2\kappa_*^2,
\end{align*}
where the first line follows by H\"older's inequality and the last line by Assumption \ref{A1} for some constant $C>0$ only depending on $\epsilon$.

Next, we need to bound $\norm{\zeta_k - \tilde{\zeta}_k}_{L(\frac{5(1+\epsilon)}{5+\epsilon})}$. For the sake of presentation clearness, we denote $a_k := \norm{\Xb_k}$ and $ \tilde{a}_k := \norm{\tilde{\Xb}_k}$, and rewrite
\begin{align}\label{trunpar}
&\norm{\zeta_k- \tilde{\zeta}_k}_{L(\frac{5(1+\epsilon)}{5+\epsilon})} \nonumber\\
=& \bigg\|M\bigg\lvert \frac{1}{a_k}- \frac{1}{\tilde{a}_k}\bigg\rvert\mathbf{1}_{\{a_k > M, \tilde{a}_k>M\}}
+ \bigg(1-\frac{M}{a_k}\bigg)\mathbf{1}_{\{a_k > M, \tilde{a}_k\leq M\}}
+\bigg(1-\frac{M}{\tilde{a}_k}\bigg)\mathbf{1}_{\{a_k \leq M, \tilde{a}_k>M\}}
\bigg\|_{L(\frac{5(1+\epsilon)}{5+\epsilon})}\nonumber\\
\leq & \bigg\|M\bigg\lvert \frac{1}{a_k}- \frac{1}{\tilde{a}_k}\bigg\rvert\mathbf{1}_{\{a_k > M, \tilde{a}_k>M\}}\bigg\|_{L(\frac{5(1+\epsilon)}{5+\epsilon})}+
\bigg\| \bigg(1-\frac{M}{a_k}\bigg)\mathbf{1}_{\{a_k > M, \tilde{a}_k\leq M\}}\bigg\|_{L(\frac{5(1+\epsilon)}{5+\epsilon})}\nonumber\\
&+\bigg\|\bigg(1-\frac{M}{\tilde{a}_k}\bigg)\mathbf{1}_{\{a_k \leq M, \tilde{a}_k>M\}}
\bigg\|_{L(\frac{5(1+\epsilon)}{5+\epsilon})},
\end{align}
where the last inequality follows by the fact that $\norm{\cdot}_{L(\frac{5(1+\epsilon)}{5+\epsilon})}$  is a norm for $\epsilon > 0$.

For the first term, we have
\begin{align*}
\bigg\|M\bigg\lvert \frac{1}{a_k}- \frac{1}{\tilde{a}_k}\bigg\rvert\mathbf{1}_{\{a_k > M, \tilde{a}_k>M\}}\bigg\|_{L(\frac{5(1+\epsilon)}{5+\epsilon})}
&= \bigg\|M\bigg\lvert \frac{\tilde{a}_k-a_k}{a_k\tilde{a}_k}
\bigg\rvert\mathbf{1}_{\{a_k > M, \tilde{a}_k>M\}}\bigg\|_{L(\frac{5(1+\epsilon)}{5+\epsilon})}\\
&\leq \frac{1}{M}\{\E|\tilde{a}_k-a_k|^{\frac{5(1+\epsilon)}{5+\epsilon}}\}^{\frac{5+\epsilon}{5(1+\epsilon)}} \\
&\leq \frac{1}{M}\{\E\norm{\Xb_k- \tilde{\Xb}_k}^{\frac{5(1+\epsilon)}{5+\epsilon}}\}^{\frac{5+\epsilon}{5(1+\epsilon)}}\\
& \leq C\gamma_1\kappa_1\kappa_*\exp\{-\gamma_2(k-1)\}/M,
\end{align*}
where the last inequality is followed by Lemma \ref{lem:tau} for some constant $C>0$ only depending on $\epsilon$. With the chosen $M \geq C \gamma_1\kappa_1\kappa_*$, we have $$\bigg\|M\bigg\lvert \frac{1}{a_k}- \frac{1}{\tilde{a}_k}\bigg\rvert\mathbf{1}_{\{a_k > M, \tilde{a}_k>M\}}\bigg\|_{L(\frac{5(1+\epsilon)}{5+\epsilon})} \leq  \exp\{-\gamma_2(k-1)\}.$$

For the second term, taking any $\epsilon_k >0$, we have
\begin{align*}
&\bigg\| \bigg(1-\frac{M}{a_k}\bigg)\mathbf{1}_{\{a_k > M, \tilde{a}_k\leq M\}}\bigg\|_{L(\frac{5(1+\epsilon)}{5+\epsilon})}\\
= &\bigg\| \bigg(1-\frac{M}{M + \epsilon_k}\bigg)\mathbf{1}_{\{M<a_k\leq M + \epsilon_k, \tilde{a}_k\leq M\}}\bigg\|_{L(\frac{5(1+\epsilon)}{5+\epsilon})} + \bigg\| \bigg(1-\frac{M}{a_k}\bigg)\mathbf{1}_{\{a_k>M + \epsilon_k, \tilde{a}_k\leq M\}}\bigg\|_{L(\frac{5(1+\epsilon)}{5+\epsilon})}\\
\leq &\frac{\epsilon_k}{M}+ \bigg\| \mathbf{1}_{\{a_k>M + \epsilon_k, \tilde{a}_k\leq M\}}\bigg\|_{L(\frac{5(1+\epsilon)}{5+\epsilon})} \leq \frac{\epsilon_k}{M}+ \{\P(|a_k-\tilde{a}_k|> \epsilon_k)\}^{\frac{5+\epsilon}{5(1+\epsilon)}}.
\end{align*}
By Markov inequality and Lemma \ref{lem:tau}, we have
$$\P(|a_k-\tilde{a}_k|> \epsilon_k) \leq \frac{\E\norm{\Xb_k - \tilde{\Xb}_k}}{ \epsilon_k} \leq \frac{C\gamma_1\kappa_1\kappa_*\exp\{-\gamma_2(k-1)\}}{ \epsilon_k}$$
for some constant $C>0$ only depending on $\epsilon$. Taking $\epsilon_k = C\gamma_1\kappa_1\kappa_*\exp\{-\frac{5+\epsilon}{6\epsilon+10}\gamma_2(k-1)\}$, we obtain
$$\bigg\| \bigg(1-\frac{M}{a_k}\bigg)\mathbf{1}_{\{a_k > M, \tilde{a}_k\leq M\}}\bigg\|_{L(\frac{5(1+\epsilon)}{5+\epsilon})}\leq 2\exp\bigg\{-\frac{5+\epsilon}{6\epsilon+10}\gamma_2(k-1)\bigg\}.$$
The third term follows by symmetry. Putting together, we have for $k >0$,
\begin{align*}
&\norm{\zeta_k- \tilde{\zeta}_k}_{L(\frac{5(1+\epsilon)}{5+\epsilon})}  \leq C\exp\bigg\{-\frac{5+\epsilon}{6\epsilon+10}\gamma_2(k-1)\bigg\},\\
&\norm{\E\Xb_0\tilde{\Xb}_k\zeta_0(\zeta_k-\tilde{\zeta}_k)}\leq C\kappa_1^2\kappa_*^2\exp\bigg\{-\frac{5+\epsilon}{6\epsilon+10}\gamma_2(k-1)\bigg\},\\
&\norm{\E\Xb_0^M\Xb_k^M - \E\Xb_0^M \E\tilde{\Xb}_k^M} \leq C
\kappa_1^2\{\kappa_1\kappa_*\gamma_1 + \kappa_*^2(\gamma_3+1)\}\exp\bigg\{-\min\Big( \frac{5+\epsilon}{6\epsilon+10}\gamma_2, \gamma_4 \Big)(k-1)\bigg\}
\end{align*}
for some constant $C>0$ only depending on $\epsilon$.
Hence for any $K \subseteq \{1,\dots, n\}$,
\begin{align*}
&\frac{1}{\card{(K)}}\lambda_{\max}\bigg\{\E\bigg(\sum_{i \in K} \Xb_i^M - \E \Xb_i^M\bigg)^2\bigg\} \\
\leq &\frac{1}{\card{(K)}}\bigg\|\sum_{i,j \in K} \E(\Xb_i^M - \E \Xb_i^M)(\Xb_j^M - \E \Xb_j^M)\bigg\| \\
\leq &\frac{1}{\card{(K)}}\sum_{i,j \in K}\norm{\E(\Xb_i^M - \E \Xb_i^M)(\Xb_j^M - \E \Xb_j^M)}\\
\leq &C\Big[ \kappa_1^4 + \kappa_*^2\kappa_1^2 +      \frac{\kappa_1^2\{\kappa_1\kappa_*\gamma_1 + \kappa_*^2(\gamma_3+1)\}\}}{\card(K)}\sum_{i,j \in K, i \neq j}\exp\Big\{-\min\Big( \frac{5+\epsilon}{6\epsilon+10}\gamma_2, \gamma_4 \Big)(|i-j|-1)\Big\}\Big]\\
\leq& C\frac{  \kappa_1^2 \{  \kappa_1^2 + \kappa_1\kappa_*\gamma_1 + \kappa_*^2(\gamma_3 + 2)\}  }{1-\exp\{-\min( \frac{5+\epsilon}{6\epsilon+10}\gamma_2, \gamma_4 )\}}
\end{align*}
for some constant $C>0$ only depending on $\epsilon$.

Similar arguments apply to $\nu^2_{\Zb^M}$ so we omit the details. This completes the proof.
\end{proof}


\begin{proof}[Proof of Lemma \ref{lem:trun}]
	Fix $\bu_1, \bu_2, \bv_1, \bv_2 \in \R^p$ with unit length and $\sigma_u \geq 0$. For any $\sigma_v \geq \sigma_u$, we perform singular value decomposition for matrix $\Xb(\sigma_v) := \sigma_u \bu_1 \bu_2^\T - \sigma_v \bv_1 \bv_2^\T$. According to Equation (8) in \cite{brand2006fast}, the non-zero singular values of $\Xb(\sigma_v)$ are identical to those of 
	$$\mathbf{S}(\sigma_v) = \begin{bmatrix}
	\sigma_u  - \sigma_v \bu_1^\T \bv_1 \bv_2^\T \bu_2 & -\sigma_v\bu_1^\T\bv_1 \norm{\bv_2 - \bu_2\bu_2^\T\bv_2}_2\\
	\sigma_v\bu_2^\T\bv_2 \norm{\bv_1 - \bu_1\bu_1^\T\bv_1}_2 & \sigma_v^2\norm{\bv_1 - \bu_1\bu_1^\T\bv_1}_2\norm{\bv_2 - \bu_2\bu_2^\T\bv_2}_2\\
	\end{bmatrix}.$$
	For simplicity, denote $w = \bu_1^\T \bv_1 \bv_2^\T \bu_2$, $\tilde{\bv}_1 = \bv_1 - \bu_1\bu_1^\T\bv_1$, $\tilde{\bu}_1 = \bv_2 - \bu_2\bu_2^\T\bv_2$. Hence $\mathbf{S}(\sigma_v)$ could be rewritten as
	$$\mathbf{S}(\sigma_v) = \begin{bmatrix}\label{sglval}
	\sigma_u - \sigma_v w & -\sigma_v\bu_1^\T\bv_1 \norm{\tilde{\bv}_2}_2\\
	\sigma_v\bu_2^\T\bv_2 \norm{\tilde{\bv}_1}_2 & \sigma_v^2\norm{\tilde{\bv}_1}_2\norm{\tilde{\bv}_2}_2\\
	\end{bmatrix}.$$
	
	Using the calculation on Page 86 in \cite{blinn1996consider}, $\norm{\mathbf{S}(\sigma_v)} = Q(\sigma_v) + R(\sigma_v)$, where 
	\begin{align*}
	Q(\sigma_v) &:= \sqrt{(\sigma_u - \sigma_vw + \sigma_v\norm{\tilde{\bv}_1}_2\norm{\tilde{\bv}_2}_2)^2 + \sigma_v^2(\bu_1^\T\bv_1 \norm{\tilde{\bv}_2}_2 + \bu_2^\T\bv_2 \norm{\tilde{\bv}_1}_2)^2}/2,\\
	R(\sigma_v) &:=  \sqrt{(\sigma_u - \sigma_vw - \sigma_v\norm{\tilde{\bv}_1}_2\norm{\tilde{\bv}_2}_2)^2 + \sigma_v^2(\bu_1^\T\bv_1 \norm{\tilde{\bv}_2}_2 - \bu_2^\T\bv_2 \norm{\tilde{\bv}_1}_2)^2}/2.
	\end{align*}
	
	We are left to show that both $Q$ and $R$ are non-deceasing function of $\sigma_v\in[\sigma_u,\infty]$. By differentiating $Q, R$ with respect to $\sigma_v$, we obtain
	\begin{align*}
	\frac{dQ}{d\sigma_v} &= c_Q(\sigma_v)[\sigma_u(\norm{\tilde{\bv}_1}_2\norm{\tilde{\bv}_2}_2 - w)  + \sigma_v\{w^2 +\norm{\tilde{\bv}_1}^2_2\norm{\tilde{\bv}_2}^2_2 + (\bu_1^\T\bv_1)^2 \norm{\tilde{\bv}_2}^2_2 +(\bu_2^\T\bv_2)^2 \norm{\tilde{\bv}_1}_2^2 \}],\\
	\frac{dR}{d\sigma_v} &= c_Q(\sigma_v)[-\sigma_u(\norm{\tilde{\bv}_1}_2\norm{\tilde{\bv}_2}_2 + w)  + \sigma_v\{w^2 +\norm{\tilde{\bv}_1}^2_2\norm{\tilde{\bv}_2}^2_2 + (\bu_1^\T\bv_1)^2 \norm{\tilde{\bv}_2}^2_2 +(\bu_2^\T\bv_2)^2 \norm{\tilde{\bv}_1}_2^2 \}]
	\end{align*}
	for some nonnegative constants $c_Q(\sigma_v),\ c_R(\sigma_v)$. 
	
	By simple algebra, we have $w^2 +\norm{\tilde{\bv}_1}^2_2\norm{\tilde{\bv}_2}^2_2 + (\bu_1^\T\bv_1)^2 \norm{\tilde{\bv}_2}^2_2 +(\bu_2^\T\bv_2)^2 \norm{\tilde{\bv}_1}_2^2 = 1$ so that 
	\[
	\frac{dQ}{d\sigma_v} = c_Q(\sigma_v)[\sigma_u(\norm{\tilde{\bv}_1}_2\norm{\tilde{\bv}_2}_2 - w)  + \sigma_v].
	\]
	Moreover, since $\bu_1, \bu_2, \bv_1, \bv_2 \in \R^p$ are all length 1, we have $|w| \leq 1$ by Cauchy-Schwartz. Hence 
	by the fact that $\sigma_v \geq \sigma_u \geq 0$, we have $\frac{dQ}{d\sigma_v} \geq 0$. On the other hand, denote $a :=\bu_1^\T\bv_1 $ and $b:= \bu_2^\T\bv_2 $ and again by Cauchy-Schwartz we have $|a| \leq 1,\ |b| \leq 1$. In addition, we have
	\begin{align*}
	\norm{\tilde{\bv}_1}_2 
	&= \sqrt{(\bv_1-\bu_1\bu_1^\T\bv_1)^\T(\bv_1-\bu_1\bu_1^\T\bv_1)} \\
	& = \sqrt{\bv_1^\T\bv_1 - \bv_1^\T\bu_1 \bu_1^\T\bv_1 - \bv_1^\T\bu_1\bu_1^\T\bv_1 + \bv_1^\T\bu_1 \bu_1^\T\bu_1\bu_1^\T\bv_1}\\
	&=\sqrt{1-a^2}.
	\end{align*}
	Similarly, we have $\norm{\tilde{\bv}_2}_2  = \sqrt{1-b^2}$. Then
	\begin{align*}
	\frac{dR}{d\sigma_v} &= c_Q(\sigma_v)\{\sigma_v-\sigma_u(\norm{\tilde{\bv}_1}_2\norm{\tilde{\bv}_2}_2 + w)\}\\
	&\geq c_Q(\sigma_v)\sigma_u(1-\norm{\tilde{\bv}_1}_2\norm{\tilde{\bv}_2}_2 - w)\\
	&\geq c_Q(\sigma_v)\sigma_u(1-\sqrt{(1-a^2)(1-b^2)} -ab).
	\end{align*}
	Since $(1-ab)^2\geq (1-a^2)(1-b^2)$ and $|ab|\leq 1$, we obtain $	\frac{dR}{d\sigma_v} \geq 0$. Therefore we have shown that $\norm{\mathbf{S}(\sigma_v)} = Q(\sigma_v) + R(\sigma_v)$ is a non-decreasing function with respect to $\sigma_v$. 
	
	Obviously $\norm{M \bv_1 \bv_2^\T - M \bu_1 \bu_2^\T} \leq \norm{\sigma_u \bv_1 \bv_2^\T - \sigma_u \bu_1 \bu_2^\T}$ since $0 <M \leq \sigma_u$. Applying the monotonicity property proved above, we have $\norm{\sigma_u \bv_1 \bv_2^\T - \sigma_u \bu_1 \bu_2^\T} \leq \norm{\sigma_u \bv_1 \bv_2^\T - \sigma_v \bu_1 \bu_2^\T}$. This completes the proof.
\end{proof}


\begin{proof}[Proof of Lemma \ref{lem:gau1}]
	By the observation in the proof of Proposition \ref{prop:gau}, we have
	$$\E\norm{\hat{\mathbf{\Sigma}}_0- \mathbf{\Sigma}_0} \leq \frac{2}{n} \E\norm{\Yb \tilde{\Yb}^{\T}} = \frac{2}{n} \E\Big(\sup_{\bu,\bv \in \mathbb{S}^{p-1}} \sum_{k = 1}^{n} \bu^{\T}\bY_k\tilde{\bY}_k^{\T}\bv\Big ) :=  \frac{2}{n} \E\Big(\sup_{\bu,\bv \in \mathbb{S}^{p-1}} W_{\bu,\bv} \Big ).$$
	Now consider
	\begin{align*}
	(W_{\bu,\bv} - W_{\bu', \bv'})^2 
	= &\Big(\sum_{k = 1}^{n} \bu^{\T}\bY_k\tilde{\bY}_k^{\T}\bv - \sum_{k = 1}^{n} {\bu'} ^{\T}\bY_k\tilde{\bY}_k^{\T}\bv'\Big)^2\\
	= &\Big(\sum_{k = 1}^{n} \bu^{\T}\bY_k\tilde{\bY}_k^{\T}\bv - \sum_{k = 1}^{n} {\bu'}^{\T}\bY_k\tilde{\bY}_k^{\T}\bv +\sum_{k = 1}^{n} {\bu'}^{\T}\bY_k\tilde{\bY}_k^{\T}\bv-
	\sum_{k = 1}^{n} {\bu'} ^{\T}\bY_k\tilde{\bY}_k^{\T}\bv'\Big)^2\\
	=& \Big(\sum_{k = 1}^{n} (\bu- {\bu'})^{\T}\bY_k\tilde{\bY}_k^{\T}\bv  +\sum_{k = 1}^{n} {\bu'}^{\T}\bY_k\tilde{\bY}_k^{\T}(\bv-\bv')\Big)^2\\
	\leq &2\Big(\sum_{k = 1}^{n} (\bu- {\bu'})^{\T}\bY_k\tilde{\bY}_k^{\T}\bv\Big)^2 + 2\Big(\sum_{k = 1}^{n} {\bu'}^{\T}\bY_k\tilde{\bY}_k^{\T}(\bv-\bv')\Big)^2\\
	=& 2\sum_{d = 0}^{n-1} \sum_{|j-k| =d} (\bu- {\bu'})^{\T}\bY_j \cdot (\bu- {\bu'})^{\T}\bY_k \cdot\bv^{\T}\tilde{\bY}_j\cdot \bv^{\T}\tilde{\bY}_k \\
	&+
	2\sum_{d = 0}^{n-1} \sum_{|j-k| =d} {\bu'}^{\T}\bY_j \cdot  {\bu'}^{\T}\bY_k \cdot(\bv-\bv')^{\T}\tilde{\bY}_j\cdot (\bv-\bv')^{\T}\tilde{\bY}_k.
	\end{align*}
	Now denote the conditional expectation $\E_{\tilde{\Yb}} := \E(\cdot|\tilde{\Yb})$. Then, 
	\begin{align*}
	&\E_{\tilde{\Yb}}(W_{\bu,\bv} - W_{\bu', \bv'})^2 \\		
	\leq &2 (\bu- {\bu'})^{\T} \mathbf{\Sigma}_0(\bu- {\bu'})\sum_{j = 1}^n \bv^{\T}\tilde{\bY}_j\tilde{\bY}_j^{\T} \bv
	+2\sum_{d = 1}^{n-1} (\bu- {\bu'})^{\T}(\mathbf{\Sigma}_d + \mathbf{\Sigma}_d^{\T})(\bu- {\bu'})\sum_{(j-k) =d} \bv^{\T}\tilde{\bY}_j\cdot \bv^{\T}\tilde{\bY}_k\\
	&+2 \sum_{j,k = 1}^n \bu'^{\T}\mathbf{\Sigma}_{|j-k|}\bu' \cdot (\bv - \bv')^{\T} \tilde{\bY}_j \cdot (\bv - \bv')^{\T}\tilde{\bY}_k    \\
	\leq & 2(\bu- {\bu'})^{\T}\Big(\mathbf{\Sigma}_0 + 2\sum_{d = 1}^{n-1}\tilde{\mathbf{\Sigma}}_d\Big)(\bu- {\bu'})\sum_{j = 1}^n \bv^{\T}\tilde{\bY}_j\tilde{\bY}_j^{\T} \bv + 2\Big(\norm{\mathbf{\Sigma}_0} + 2\sum_{d = 1}^{n-1} \norm{\mathbf{\Sigma}_d}\Big)\sum_{j = 1}^n (\bv-\bv')^{\T}\tilde{\bY}_j\tilde{\bY}_j^{\T} (\bv-\bv')\\
	\leq & 2\norm{\Big(\mathbf{\Sigma}_0 + 2\sum_{d = 1}^{n-1}\tilde{\mathbf{\Sigma}}_d\Big)^{\frac{1}{2}}(\bu - \bu')}^2 \norm{\tilde{\Yb}}^2
	+2\Big(\norm{\mathbf{\Sigma}_0} + 2\sum_{d = 1}^{n-1} \norm{\mathbf{\Sigma}_d}\Big)\norm{(\bv-\bv')^{\T}\tilde{\Yb}}^2,
	\end{align*}
	where the second inequality is followed by defining $\tilde{\mathbf{\Sigma}}_d := (\Ub_d \mathbf{\Lambda}_d \Ub_d^{\T} +  \Vb_d \mathbf{\Lambda}_d \Vb_d^{\T})/2$. Here $\Ub_d,  \Vb_d, \mathbf{\Lambda}_d$ are left singular vectors, right singular vectors and singular values of $\mathbf{\Sigma}_d$ for all $1 \leq d\leq n-1$. Note that $\tilde{\mathbf{\Sigma}}_d$ are symmetric and positive semidefinite for all $d$, and hence so is $\mathbf{\Sigma}_0 + 2\sum_{d = 1}^{n-1}\tilde{\mathbf{\Sigma}}_d$. 
	
	Define the following Gaussian process:
	$$Y_{\bu,\bv} :=  \sqrt{2}\norm{\tilde{\Yb}}\bu^{\T}\Big(\mathbf{\Sigma}_0 + 2\sum_{d = 1}^{n-1}\tilde{\mathbf{\Sigma}}_d\Big)^{\frac{1}{2}}\bg + \sqrt{2}\Big(\norm{\mathbf{\Sigma}_0} + 2\sum_{d = 1}^{n-1} \norm{\mathbf{\Sigma}_d}\Big)^{\frac{1}{2}} \bv^{\T} \tilde{\Yb}\bg',$$
	where $\bg, \bg'$ are independent standard Gaussian random vectors in $\R^p$ and $\R^n$ respectively. Thus by previous inequality, we have
	$$\E_{\tilde{\Yb}}(W_{\bu,\bv} - W_{\bu', \bv'})^2 \leq \E_{\tilde{\Yb}}(Y_{\bu,\bv} - Y_{\bu', \bv'})^2.$$
	Hence by Slepian-Fernique inequality \citep{slepian1962one}, we have
	\begin{align*}
	&\E_{\tilde{\Yb}} \sup_{\bu,\bv \in \mathbb{S}^{p-1}} W_{\bu,\bv} \\
	\leq& \E_{\tilde{\Yb}} \sup_{\bu,\bv \in \mathbb{S}^{p-1}} Y_{\bu,\bv} \\
	=& \sqrt{2} \norm{\tilde{\Yb}} \cdot\E\sup_{\bu \in \mathbb{S}^{p-1}}\bu^{\T}\Big(\mathbf{\Sigma}_0 + 2\sum_{d = 1}^{n-1}\tilde{\mathbf{\Sigma}}_d\Big)^{\frac{1}{2}}\bg
	+ \sqrt{2}\Big(\norm{\mathbf{\Sigma}_0} + 2\sum_{d = 1}^{n-1} \norm{\mathbf{\Sigma}_d}\Big)^{\frac{1}{2}}	\cdot\E_{\tilde{\Yb}} \sup_{\bv \in \mathbb{S}^{p-1}}\bv^{\T} \tilde{\Yb} \bg'\\
	\leq &\sqrt{2}  \norm{\tilde{\Yb}} \cdot \E\norm{\Big(\mathbf{\Sigma}_0 + 2\sum_{d = 1}^{n-1}\tilde{\mathbf{\Sigma}}_d\Big)^{\frac{1}{2}}\bg} + \sqrt{2}\Big(\norm{\mathbf{\Sigma}_0} + 2\sum_{d = 1}^{n-1} \norm{\mathbf{\Sigma}_d}\Big)^{\frac{1}{2}}\cdot \E_{\tilde{\Yb}}\norm{ \tilde{\Yb} \bg'}\\
	\leq & \sqrt{2}  \norm{\tilde{\Yb}}  \cdot \sqrt{\tr\bigg(\mathbf{\Sigma}_0 + 2\sum_{d = 1}^{n-1}\tilde{\mathbf{\Sigma}}_d\bigg)} + \sqrt{2}\Big(\norm{\mathbf{\Sigma}_0} + 2\sum_{d = 1}^{n-1} \norm{\mathbf{\Sigma}_d}\Big)^{\frac{1}{2}}\cdot \sqrt{\tr(\tilde{\Yb} \tilde{\Yb}^{\T})}.
	\end{align*}
	Taking expectation with respect to $\tilde{\Yb}$ and using the fact that $\tilde{\Yb}$ is an independent copy of $\Yb$, we obtain
	\begin{align*}
	\E \sup_{\bu,\bv \in \mathbb{S}^{p-1}} W_{\bu,\bv}  \leq \sqrt{2}  \E\norm{\Yb}  \cdot \sqrt{\tr\bigg(\mathbf{\Sigma}_0 + 2\sum_{d = 1}^{n-1}\tilde{\mathbf{\Sigma}}_d\bigg)} + \sqrt{2}\sqrt{\norm{\mathbf{\Sigma}_0} + 2\sum_{d = 1}^{n-1} \norm{\mathbf{\Sigma}_d}}\cdot \sqrt{n\tr(\mathbf{\Sigma}_0)}.
	\end{align*}
	This completes the proof of Lemma \ref{lem:gau1}.
\end{proof}

\begin{proof}[Proof of Lemma \ref{lem:gau2}]
	Define $W_{\bu, \bv}:= \bu^{\T}\Yb\bv $. Then,
	\begin{align*}
	\E (W_{\bu, \bv} - W_{\bu', \bv'})^2 
	=& \E (\bu^{\T}\Yb\bv - \bu'^{\T}\Yb\bv')^2\\
	\leq &2\E ((\bu - \bu')^{\T}\Yb\bv)^2 + 2\E (\bu'^{\T}\Yb(\bv-\bv'))^2\\
	= &2\sum_{i,j} (\bu - \bu')^{\T}\mathbf{\Sigma}_{|i-j|}(\bu - \bu')\bv_i\bv_j
	+ 2 \sum_{i,j} \bu'^{\T}\mathbf{\Sigma}_{|i-j|}\bu' (\bv_i-\bv_i^{'})(\bv_j - \bv_j^{'}).
	\end{align*}
	
	In addition, define
	\begin{align*}
	&\mathbf{\Sigma}_L := \left[ {\begin{array}{cccc}
		\mathbf{\Sigma}_0 & \mathbf{\Sigma}_1 & \cdots & \mathbf{\Sigma}_{n-1}\\
		\mathbf{\Sigma}_1^{\T}  & \mathbf{\Sigma}_0 & \cdots &  \mathbf{\Sigma}_{n-2}\\
		\cdots &  \cdots & \cdots & \cdots \\
		\mathbf{\Sigma}_{n-1}^{\T} & \mathbf{\Sigma}_{n-2}^{\T} & \cdots & \mathbf{\Sigma}_{0}\\
		\end{array} } \right],
	\mathbf{\Sigma}_{L,\bu} := \left[ {\begin{array}{cccc}
		\bu^{\T} & \mathbf{0} & \cdots & \mathbf{0}\\
		\mathbf{0} & \bu^{\T} & \cdots & \mathbf{0}\\
		\cdots &  \cdots & \cdots & \cdots \\
		\mathbf{0} & \mathbf{0} & \cdots & \bu^{\T}\\
		\end{array} } \right] \mathbf{\Sigma}_L 
	\left[ {\begin{array}{cccc}
		\bu & \mathbf{0} & \cdots & \mathbf{0}\\
		\mathbf{0} & \bu& \cdots & \mathbf{0}\\
		\cdots &  \cdots & \cdots & \cdots \\
		\mathbf{0} & \mathbf{0} & \cdots & \bu\\
		\end{array} } \right],\\
	&\mathbf{\Sigma}^{\circ}_{\bu} := (\bu^{\T}\mathbf{\Sigma}_0\bu) \mathbf{1}_n\mathbf{1}_n^{\T},
	\hspace{2.7cm} \mathbf{\Sigma}^{\diamond} := \norm{\mathbf{\Sigma}_0}\mathbf{1}_n\mathbf{1}_n^{\T}.
	\end{align*}
	Since $\mathbf{\Sigma}_L$ is a positive semi-definite matrix, we have
	$$\mathbf{\Sigma}_{L,\bu} \preceq \mathbf{\Sigma}^{\circ}_{\bu} \preceq \mathbf{\Sigma}^{\diamond}$$
	for all $\bu \in \mathbb{S}^{p-1}$, where ``$\preceq$" is the Loewner partial order of Hermitian matrices. Hence,
	\begin{align*}
	&\E (W_{\bu, \bv} - W_{\bu', \bv'})^2 
	\leq  2\norm{\Big(\mathbf{\Sigma}_0 + 2\sum_{d = 1}^{n-1}\tilde{\mathbf{\Sigma}}_d\Big)^{\frac{1}{2}}(\bu - \bu')}^2 + 
	2\norm{\mathbf{\Sigma}_0}(\bv -  \bv')^{\T}\mathbf{1}_n\mathbf{1}_n^{\T}(\bv -  \bv').
	\end{align*}
	Then define the following Gaussian process:
	$$Y_{\bu, \bv} := \sqrt{2}\bu^{\T}\Big(\mathbf{\Sigma}_0 + 2\sum_{d = 1}^{n-1}\tilde{\mathbf{\Sigma}}_d\Big)^{\frac{1}{2}}\bg + \sqrt{2}\norm{\mathbf{\Sigma}_0}^{\frac{1}{2}} \bv^{\T}\bg',$$
	where $\bg \in \R^p, \bg' \in \R^n$ are independent Gaussian random vectors with mean $\mathbf{0}$ and covariance matrices $\Ib_p$ and  $\mathbf{1}_n\mathbf{1}_n^{\T}$ respectively. Thus by previous inequality, we have
	$$\E(W_{\bu,\bv} - W_{\bu', \bv'})^2 \leq \E(Y_{\bu,\bv} - Y_{\bu', \bv'})^2.$$
	Hence by Slepian-Fernique inequality, we have
	\begin{align*}
	&\E \sup_{\bu,\bv \in \mathbb{S}^{p-1}} W_{\bu,\bv} 
	\leq \E \sup_{\bu,\bv \in \mathbb{S}^{p-1}} Y_{\bu,\bv} \\
	= & \sqrt{2} \E\sup_{\bu \in \mathbb{S}^{p-1}}\bu^{\T}\Big(\mathbf{\Sigma}_0 + 2\sum_{d = 1}^{n-1}\tilde{\mathbf{\Sigma}}_d\Big)^{\frac{1}{2}}\bg
	+ \sqrt{2}\norm{\mathbf{\Sigma}_0}^{\frac{1}{2}}	\cdot\E \sup_{\bv \in \mathbb{S}^{p-1}}\bv^{\T} \bg'\\
	\leq &\sqrt{2} \E\norm{\Big(\mathbf{\Sigma}_0 + 2\sum_{d = 1}^{n-1}\tilde{\mathbf{\Sigma}}_d\Big)^{\frac{1}{2}}\bg} + \sqrt{2}\norm{\mathbf{\Sigma}_0}^{\frac{1}{2}}\cdot \E\norm{ \bg'}\\
	\leq & \sqrt{2} \sqrt{\tr\bigg(\mathbf{\Sigma}_0 + 2\sum_{d = 1}^{n-1}\tilde{\mathbf{\Sigma}}_d\bigg)} + \sqrt{2}\norm{\mathbf{\Sigma}_0}^{\frac{1}{2}} \cdot \sqrt{n}.
	\end{align*}
	This completes the proof of Lemma \ref{lem:gau2}.
\end{proof}

\subsection{Proof of results in Section \ref{sec:app}}\label{subsec:pfsec3}

\begin{proof}[Proof of Theorem \ref{ex1}]
	We first examine Assumptions \ref{A1} and \ref{A4}. First of all, we will study VAR($1$) model, i.e., $\bY_t = \Ab \bY_{t-1} + \bE_t$. Notice that for VAR($1$), we could rewrite the original sequence as a moving-average model, i.e., $\bY_t = \sum_{j = 0}^\infty \Ab^j\bE_{t-j}$. For any $\bu\in \R^p$, we have
	\begin{align*}
	\norm{\bu^\T \bY_t}_{\psi_2} &= \Big\|\sum_{j = 0}^\infty \bu^\T\Ab^j\bE_{t-j}\Big\|_{\psi_2}\\
	&\leq C\Big(\sum_{j = 0}^\infty \norm{\bu^\T\Ab^j\bE_{t-j}}_{\psi_2}^2 \Big)^{\frac{1}{2}}\\
	&\leq Cc'\Big(\sum_{j = 0}^\infty \norm{\bu^\T\Ab^j\bE_{t-j}}_{L(2)}^2 \Big)^{\frac{1}{2}}= Cc'\norm{\bu^\T\bY_t}_{L(2)}
	\end{align*}
	for some universal constant $C>0$. Here the second line and last equality are followed by the fact that $\{\bE_t\}_{t \in \Z}$ is a sequence of independent random vector, and the third line by the moment assumption on $\{\bE_t\}_{t \in \Z}$. Since $\bY_{t-1}$ is a stable process when $\norm{\Ab} <1$, $\norm{\bu^\T \bY_t}_{\psi_2}\leq Cc'\norm{\bu^\T\bY_t}_{L(2)} <\infty$ for all $\bu \in \R^p$. 
	
	Denote $\bar{\bY}_t := (\bY_t^\T \dots \bY_{t-d}^\T)^\T$ and $\bar{\bE}_t := (\bE_t^\T ~\mathbf{0}^\T \dots \mathbf{0}^\T)^\T$. For $\{\bY_t\}_{t \in \Z}$ generated from a VAR($d$) model, $\{\bar{\bY}_t\}_{t \in \Z}$ is a VAR($1$) process, i.e., $\bar{\bY}_t  = \bar{\Ab}\cdot\bar{\bY}_{t-1} + \bar{\bE}_t$. Thus by previous argument, taking any $\bv \in \R^{p(d+1)}$ where only the first $p$ digits are non-zero and denoting $\bv' \in \R^p$ to be first-$p$ part of $\bv$, we have $\norm{\bv'^\T \bY_t}_{\psi_2} = \norm{\bv^\T \bar{\bY}_t}_{\psi_2} \leq  C\norm{\bv^\T \bar{\bY}_t}_{L(2)} = C\norm{\bv'^\T \bar{\bY}_t}_{L(2)} <\infty$ for some constant $C>0$ only depending on $c'$ where the last inequality is followed by the fact that $\{\bY_t\}$ is a stable process (see Lemma \ref{lem:rhoA}).  Assumptions \ref{A1} and \ref{A4} are verified.
	
	Then we examine Assumption \ref{A2}. Without loss of generality, take $j = 0$ in Assumption \ref{A2}. Let $\{\tilde{\bY}_t\}_{t = 1-d}^0$ be a sequence of random vectors independent of $\{\bY_t\}_{t \leq 0}$ and identically distributed as $\{\bY_{t}\}_{t = 1-d}^0$. Define $\tilde{\bY}_t = \Ab_1\tilde{\bY}_{t-1}+ \dots+\Ab_d\tilde{\bY}_{t-d} + \bE_t$ for every $t >0$. It is obvious that $\{\tilde{\bY}_t\}_{t>0}$ is independent of $\{\tilde{\bY}_t\}_{t \leq 0}$ and identically distributed as $\{\bY_t\}_{t > 0}$. Moreover, for any $t \geq 1$, we have
	\begin{align*}
	\norm{\norm{\bY_t - \tilde{\bY}_t}_2}_{L(1+\epsilon)}
	&=\{\E\norm{\Ab_1\bY_{t-1}+ \dots+\Ab_d\bY_{t-d} +\bE_t- (\Ab_1\tilde{\bY}_{t-1}+ \dots+\Ab_d\tilde{\bY}_{t-d}+\bE_t)}_2^{1+\epsilon  }\}^{\frac{1}{1+\epsilon}} \\
	& \leq  \{\E\norm{\Ab_1(\bY_{t-1}-\tilde{\bY}_{t-1})+ \dots+\Ab_d(\bY_{t-d}-\tilde{\bY}_{t-d}) }_2^{ 1+\epsilon }\}^{\frac{1}{1+\epsilon}} \\
	& \leq \sum_{k= 1}^d a_k \{\E\norm{\bY_{t-k} - \tilde{\bY}_{t-k}}_2^{ 1+\epsilon  }\}^{\frac{1}{1+\epsilon}},
	\end{align*}
	where the third line follows by $\norm{\cdot}_{L(1+\epsilon)}$ is a norm for $\epsilon > 0$.
	Denoting $\phi_t= \norm{\norm{\bY_t - \tilde{\bY}_t}_2}_{L(1+\epsilon)}$, we have $\phi_t \leq \sum_{k= 1}^d a_k\phi_{t-k}$. Let $\bv$ be the unit vector with 1 at first position and 0 elsewhere. Then by iteration, we have
	$$\bv^{\T}(\phi_t, \dots, \phi_{t-d+1})^{\T} \leq \bv^{\T}\bar{\Ab}^t(\phi_0, \dots, \phi_{1-d})^{\T}  \leq \norm{\bar{\Ab}^t}\norm{(\phi_0, \dots, \phi_{1-d})^{\T}}_2.$$
	Note that $\phi_t= C\kappa_*$ for $t \leq 0$ by Assumption \ref{A1} for some constant $C>0$ only depending on $\epsilon$. By the following lemma which provides sufficient and necessary conditions for matrix $\bar{\Ab}$ to have spectral radius strictly less than 1, we could choose some arbitary $\rho_1$ such that $\rho(\bar{\Ab})< \rho_1<1$. 
	\begin{lemma}\label{lem:rhoA}
		For $\bar{\Ab}$ defined above, $\rho(\bar{\Ab}) <1$ if and only if $\sum_{k = 1}^d a_k <1$, where $\rho(\bar{\Ab})$ is the spectral radius of $\bar{\Ab}$.
	\end{lemma}

\begin{proof}[Proof of Lemma \ref{lem:rhoA}]
	The result is well known and here we include a proof merely for completeness. First of all, we prove the sufficient condition. A key observation is that the characteristic equation $\det(\bar{\Ab} -\lambda\Ib_d) = 0$ for matrix $\bar{\Ab}$ is 
	$$f(\lambda)=\lambda^d - a_1\lambda^{d-1} - \dots - a_{d-1}\lambda^1 - a_d = 0.$$
	Assume $\sum_{j= 1}^d a_j \geq 1$. We obtain $f(1) = 1-\sum_{j = 1}^d a_j \leq 0$ and $f(\infty) = \infty$. By continuity of $f(\lambda)$, there exists at least one root whose modulus is greater than or equal to 1. This contradicts with the fact that $\rho(\bar{\Ab})$ is strictly less than 1.
	
	Secondly, we prove the necessary condition. Suppose there exists a root $z \in \mathbb{C}$ (the set of complex numbers) of $f(\lambda)$ such that $|z| \geq 1$. Here $|z|$ is the modulus of $z$. Then
	$$|z|^d = |a_1z^{d-1} + \dots + a_{d-1}z^1 +a_d| \leq a_1|z|^{d-1} + \dots + a_{d-1}|z|^1 + a_d.$$
	Since $|z| \geq 1$, we have $|z|^k \leq |z|^{d}$ for $0 \leq k \leq d-1$. Hence 
	$|z|^d \leq (a_1+\dots + a_d)|z|^d$ implies $a_1+\dots + a_d \geq 1.$
	This contradicts the fact that $\sum_{j=1}^d a_j$ is strictly less than 1. This completes the proof.
\end{proof}

	 By Gelfand's formula, there exists a $K >0$, such that for all $t \geq K$, $\norm{\bar{\Ab}^t} < \rho_1^t$. For $t < K$, we have
	$$\phi_t \leq 2d\kappa_*\bigg(\frac{\norm{\bar{\Ab}}}{\rho_1}\bigg)^K\rho_1^t.$$
	For $t \geq K$, we have $\phi_t \leq Cd\kappa_*\rho_1^t$ for some constant $C>0$ only depending on $\epsilon$. Taking $\gamma_1 = Cd(\kappa_*/\kappa_1)(\norm{\bar{\Ab}}/\rho_1)^K$ for some constant $C>0$ only depending on $\epsilon$ and $\gamma_2 = \log(\rho^{-1}_1)$ verifies Assumption \ref{A2}.
	
	Lastly, we verify Assumption \ref{A3}. Following the same construction as in verifying Assumption \ref{A2}, we have for any $\bu \in \mathbb{S}^{p-1}$,
	\begin{align*}
	&\norm{(\bY_t - \tilde{\bY}_t)^\T\bu}_{L(1+\epsilon)}\\
	=&( \E|\{
	\Ab_1\bY_{t-1}+ \dots+\Ab_d\bY_{t-d} +\bE_t- (\Ab_1\tilde{\bY}_{t-1}+ \dots+\Ab_d\tilde{\bY}_{t-d}+\bE_t)
	\}^\T\bu|^{1+\epsilon})^{\frac{1}{1+\epsilon}} \\
	\leq & ( \E|\{
	\Ab_1\bY_{t-1}+ \dots+\Ab_d\bY_{t-d} - (\Ab_1\tilde{\bY}_{t-1}+ \dots+\Ab_d\tilde{\bY}_{t-d})
	\}^\T\bu|^{1+\epsilon})^{\frac{1}{1+\epsilon}} \\
	\leq &\sum_{k = 1}^d a_k\{\E|(\bY_{t-k}- \tilde{\bY}_{t-k})^\T\bu_k |^{1+\epsilon}\}^{\frac{1}{1+\epsilon}},
	\end{align*}
	for $\bu_k :=\Ab_k\bu/\norm{\Ab_k\bu}_2$, $k \in \{1,\dots, d\}$. The result follows as we follow the same arguments to verify Assumption \ref{A2}. This completes the proof of Theorem \ref{ex1}.
\end{proof}

\begin{proof}[Proof of Theorem \ref{ex2}]
	First of all, we verify Assumptions \ref{A1} and \ref{A4}. It is trivial that Assumptions \ref{A1} and \ref{A4} are satisfied if $W_t = 0$ almost surely for all $t \in \Z$. If $W_t \neq 0$ almost surely, then for all $\bu \in \R^p$,
	$\norm{\bu^\T\bY_t}_{\psi_2} \leq \norm{W_t}_{L(\infty)}\norm{\bu^\T\bE_t}_{\psi_2} \leq c'\kappa_W\norm{\bu^\T\bE_t}_{L(2)} \leq c'\frac{\kappa_W}{\inf_{t \in \Z}\norm{W_t}_{L(2)}}\norm{\bu^\T\bY_t}_{L(2)} <\infty$. This verifies Assumptions \ref{A1} and \ref{A4}.
	
	For Assumption \ref{A2}, without loss of generality, take $j = 0$. Since $\{W_t\}_{t \in \Z}$ is a sequence of uniformly bounded $\tau$-mixing random variables, we may find $\{\tilde{W}_t\}_{t >0}$ which is independent of $\{W_t\}_{t \leq 0}$, identically distributed as $\{W_t\}_{t > 0}$, and for any $t \geq 1$,
	$$\E|\tilde{W}_t - W_t| \leq \kappa_W \gamma_5\exp\{-\gamma_6(t-1)\}.$$
	Define $\tilde{\bY}_t:= \tilde{W}_t\bE_t$ for all $t \geq 1$. It is obvious that $\{\tilde{\bY}_t\}_{t >0}$ is independent of $\{\bY_t\}_{t \leq 0}$ and identically distributed as $\{\bY_t\}_{t > 0}$. Moreover, for any integer $t\geq 1$, 
	\begin{align*}
	\norm{\norm{\bY_t - \tilde{\bY}_t}_2}_{L(1+\epsilon)}
	&\leq (\E\norm{W_t \bE_t- \tilde{W}_t\bE_t}_2^{1+\epsilon})^{\frac{1}{1+\epsilon}} \\
	&\leq (\E|W_t - \tilde{W}_t|\cdot|W_t - \tilde{W}_t|^{1+\epsilon})^{1+\epsilon}(\E\norm{\bE_t}_2^{1+\epsilon})^{\frac{1}{1+\epsilon}}\\
	&\leq C\kappa_*' \kappa_W \gamma_5^{\frac{1}{1+\epsilon}}\exp\Big\{-\frac{1}{1+\epsilon}\gamma_6(t-1)\Big\}
	\end{align*}
	for some constant $C>0$ only depending on $\epsilon$.
	Taking $\gamma_1 = C\kappa_*' \kappa_W \gamma_5^{\frac{1}{1+\epsilon}}/\kappa_1$ and $\gamma_2 = \frac{1}{1+\epsilon}\gamma_6$ verifies Assumption $\ref{A2}$.
	
	For Assumption \ref{A3}, without loss of generality, take $j = 0$. Let $\{\tilde{\bY}_t\}_{t >0}$ be the same construction as above. For any integer $t\geq 1$, 
	\begin{align*}
	\sup_{\bu \in \mathbb{S}^{p-1}}\norm{(\bY_t - \tilde{\bY}_t)^{\T}\bu}_{L(1+\epsilon)}
	&= \sup_{\bu \in \mathbb{S}^{p-1}}\{\E|(W_t\bE_t - \tilde{W}_t\bE_t)^{\T}\bu|^{1+\epsilon}\}^{\frac{1}{1+\epsilon}}\\
	& = (\E|W_t-\tilde{W}_t|^{1+\epsilon})^{\frac{1}{1+\epsilon}}
	\sup_{\bu \in \mathbb{S}^{p-1}}(\E|\bE_t^{\T}\bu|^{1+\epsilon})^{\frac{1}{1+\epsilon}}\\
	&\leq C\kappa_1' \kappa_W \gamma_5^{\frac{1}{\epsilon}}\exp\Big\{-\frac{1}{1+\epsilon}\gamma_6(t-1)\Big\}
	\end{align*}
	for some constant $C>0$ only depending on $\epsilon$. Taking $\gamma_3 = C\kappa_1' \kappa_W \gamma_5^{\frac{1}{1+\epsilon}}/\kappa_1$ and $\gamma_4 = \frac{1}{1+\epsilon}\gamma_6$ verifies Assumption $\ref{A2}$. This completes the proof of Theorem \ref{ex2}.
\end{proof}

\begin{proof}[Proof of Theorem \ref{ex3}]
	We first verify Assumptions \ref{A2} and \ref{A3}. Without loss of generality, take $j = 0$ in Assumption \ref{A2}. Let $\tilde{\bY}_0$ be a random vector independent of $\{\bY_t\}_{t \leq 0}$ and identically distributed as $\bY_0$. Define $\tilde{\bY}_t = \Ab\tilde{\bY}_{t-1} + H(\tilde{\bY}_{t-1})\bE_t$ for every $t \geq 1$. It is obvious that $\{\tilde{\bY}_t\}_{t >0}$ is independent of $\{\bY_t\}_{t \leq 0}$ and identically distributed as $\{\bY_t\}_{t >0}$. We obtain for any $t\geq 1$,
	\begin{align*}
	\norm{\norm{\bY_t - \tilde{\bY}_t}_2}_{L(1+\epsilon)}
	&=[\E\norm{\Ab\bY_{t-1}+ H(\bY_{t-1})\bE_t-\{\Ab\tilde{\bY}_{t-1} + H(\tilde{\bY}_{t-1})\bE_t\}}_2^{1+\epsilon}]^{\frac{1}{1+\epsilon}} \\
	& \leq [\E\norm{\Ab\bY_{t-1}- \Ab\tilde{\bY}_{t-1}+ \{H(\bY_{t-1})- H(\tilde{\bY}_{t-1})\}\bE_t }_2^{1+\epsilon}]^{\frac{1}{1+\epsilon}} \\
	& \leq (a_1 + a_2)\norm{\norm{\bY_{t-1} - \tilde{\bY}_{t-1}}_2}_{L(1+\epsilon)}.
	\end{align*}
	By iteration, we obtain
	\begin{align*}
	\norm{\norm{\bY_t - \tilde{\bY}_t}_2}_{L(1+\epsilon)} \leq(a_1 + a_2)^t(\E\norm{\bY_0 - \tilde{\bY}_0}_2^{1+\epsilon})^{\frac{1}{1+\epsilon}}
	\leq C\kappa_*(a_1 + a_2)^t
	\end{align*}
	for some constant $C>0$ only depending on $\epsilon$. Taking $\gamma_1 = C\kappa_*/\kappa_1$ and $\gamma_2 = -\log(a_1+a_2)$ verifies Assumption \ref{A2}.
	
	For Assumption \ref{A3}, following the construction above, we have for any $\bu \in \mathbb{S}^{p-1}$ and $t\geq 1$,
	\begin{align*}
	\norm{(\bY_t- \tilde{\bY}_t)^{\T}\bu}_{L(1+\epsilon)}
	&=[\E|\{\Ab\bY_{t-1} + H(\bY_{t-1})\bE_t- (\Ab\tilde{\bY}_{t-1} + H(\tilde{\bY}_{t-1})\bE_t) \}^{\T}\bu|^{1+\epsilon}]^{\frac{1}{1+\epsilon}} \\
	& \leq[ \E|\{
	\Ab\bY_{t-1} - \Ab\tilde{\bY}_{t-1} + (H(\bY_{t-1})-H(\bY_{t-1}))\bE_t
	\}^{\T}\bu|^{1+\epsilon}]^{\frac{1}{1+\epsilon}} \\
	& \leq a_1\norm{(\bY_{t-1}- \tilde{\bY}_{t-1})^{\T}\bv}_{L(1+\epsilon)}
	+ a_2\frac{\kappa_{1}'}{\kappa_{*}'}\norm{\norm{\bY_{t-1} - \tilde{\bY}_{t-1}}_2}_{L(1+\epsilon)},
	\end{align*}
	where $\bv:= \Ab\bu/\norm{\Ab\bu}_2 \in \mathbb{S}^{p-1}$.
	By iteration, we obtain
	$$	\norm{(\bY_t - \tilde{\bY}_t)^{\T}\bu}_{L(1+\epsilon)} \leq C\{\kappa_1a_1^t + 2\kappa_*\frac{\kappa_{1}'}{\kappa_{*}'}
	a_2\sum_{\ell = 0}^{t-1}a_1^\ell(a_1+a_2)^{t-1-\ell}\} \leq C(a_1+a_2)^t\max(\kappa_*\frac{\kappa_{1}'}{\kappa_{*}'}, \kappa_1)$$
	for some constant $C>0$ only depending on $\epsilon$. Taking $\gamma_3 = C\max(\frac{\kappa_*\kappa_{1}'}{\kappa_1\kappa_{*}'}, 1)$ and  $\gamma_4 = -\log(a_1 + a_2)$ verifies Assumption \ref{A3}.
	
	By further assuming that $\{\bY_t\}$ is a stationary process and $H(\cdot)$ is uniformly bounded, we have that for  all $t \in \Z$, $\sup_{\bu \in \mathbb{S}^{p-1}}\norm{\bu^\T \bY_t}_{\psi_2}  \leq \norm{\Ab}\sup_{\bu \in \mathbb{S}^{p-1}} \norm{\bu^\T\bY_{t-1}}_{\psi_2} +D_2\sup_{\bv \in \mathbb{S}^{p-1}} \norm{\bv^\T\bE_t}$. By stationarity, this renders $\kappa_1 = \sup_{\bu \in \mathbb{S}^{p-1}}\norm{\bu^\T \bY_t}_{\psi_2} \leq \frac{1}{1-\norm{\Ab}}D_2\kappa_1' < \infty$. Similar argument applies to $\kappa_*$. This verifies Assumption \ref{A1} under additional assumptions and completes the proof of Theorem \ref{ex3}.
\end{proof}

\section*{Appendix}\label{appendix}

\appendix

\section{Proof of Theorem \ref{thm:bern}}\label{sec:apx}

In this appendix we present the proof of Theorem \ref{thm:bern}, which slightly extends the Bernstein-type inequality proven by \cite{banna2016bernstein} in which the random matrix sequence is assumed to be $\beta$-mixing. The proof is largely identical to theirs, and we include it here mainly for completeness.


In the following, $\tau_k$ is abbreviate of $\tau(k)$ for $k \geq 1$. If a matrix $\Xb$ is positive semidefinite, denote it as $\Xb \succeq 0$.  For any $x>0$, we define $h(x) = x^{-2}(e^x - x - 1)$. Denote the floor, ceiling, and integer parts of a real number $x$ by $\lfloor x\rfloor$, $\lceil x \rceil$, and $[x]$. For any two real numbers $a,b$, denote $a\vee b := \max\{a,b\}$. Denote the exponential of matrix $\Xb$ as $\exp(\Xb) = \Ib_p + \sum_{q = 1}^{\infty}\Xb^q/q!$. Letting $\sigma_1$ and $\sigma_2$ be two sigma fields, denote $\sigma_1 \vee \sigma_2$ to be the smallest sigma field that contains $\sigma_1$ and $\sigma_2$ as sub-sigma fields.

A roadmap of this appendix is as follows. Section \ref{sec:tau} formally introduces the concept of $\tau$-mixing coefficient. Section \ref{subsec:bernover} previews the proof of Theorem \ref{thm:bern} and indicates some major differences from the proofs in \cite{banna2016bernstein}. Section \ref{sec:cantor} contains the construction of Cantor-like set which is essential for decoupling dependent matrices.  Section \ref{sec:decoupling} develops a major decoupling lemma for $\tau$-mixing random matrices and will be used in Section \ref{sec:lem2} to prove Lemma \ref{lem:bernlm2}. Then Section \ref{sec:thm-main} finishes the proof of Theorem \ref{thm:bern}.

\subsection{Introduction to $\tau$-mixing random sequence}\label{sec:tau}
	This section introduces the $\tau$-mixing coefficient. Consider $(\Omega,\cF,\P)$ to be a probability space, $X$ an $L_1$-integrable random variable taking value in a Polish space $(\mathcal{X}, \norm{\cdot}_{\cX})$,  and $\cA$ a sigma algebra of $\cF$. The $\tau$-measure of dependence between $X$ and $\cA$ is defined to be
\begin{align*}
\tau(\mathcal{A}, X;\norm{\cdot}_{\cX}) = \Big\lVert\sup_{g \in \Lambda(\norm{\cdot}_{\cX})}\Big\{\int g(x)\P_{X|\mathcal{A}}({\sf d}x) - \int g(x)\P_X({\sf d}x) \Big\}\Big\lVert_{L(1)},
\end{align*}
where $\P_X$ is the distribution of $X$, $\P_{X|\mathcal{A}}$ is the conditional distribution of $X$ given $\mathcal{A}$, and $\Lambda(\norm{\cdot}_{\cX})$ stands for the set of 1-Lipschitz functions from $\mathcal{X}$ to $\R$ with respect to the norm $\norm{\cdot}_{\cX}$. 

The following two lemmas from \cite{dedecker2004coupling} and \cite{dedecker2007book} characterize the intrinsic ``coupling property" of $\tau$-measure of dependence, which will be heavily exploited in the derivation of our results. 

\begin{lemma}[Lemma 3 in \cite{dedecker2004coupling}]\label{lem:tauup}
	Let $(\Omega, \mathcal{F}, \P)$ be a probability space, $X$ be an integrable random variable with values in a Banach space $(\mathcal{X}, \norm{\cdot}_{\cX})$ and $\mathcal{A}$ a sigma algebra of $\mathcal{F}$. If $Y$ is a random variable distributed as $X$ and independent of $\mathcal{A}$, then
	$$\tau(\mathcal{A}, X;\norm{\cdot}_{\cX}) \leq \E \norm{X-Y}_{\cX}.$$
\end{lemma}

\begin{lemma}[Lemma 5.3 in \cite{dedecker2007book}]\label{lem:dedecker}
	Let $(\Omega, \mathcal{\mathcal{F}}, \P)$ be a probability space, $\mathcal{A}$ be a sigma algebra of $\mathcal{F}$, and $X$ be a random variable with values in a Polish space $(\mathcal{X}, \norm{\cdot}_{\cX})$. Assume that $\int \norm{x-x_0}_{\cX}\P_X({\sf d}x)$ is finite for any $x_0 \in \mathcal{X}$. Assume that there exists a random variable $U$ uniformly distributed over $[0,1]$, independent of the sigma algebra generated by $X$ and $\mathcal{A}$. Then there exists a random variable $\tilde{X}$, measurable with respect to $\mathcal{A}\vee \sigma(X) \vee \sigma(U)$, independent of $\mathcal{A}$ and distributed as $X$, such that 
	\begin{align*}
	\tau(\mathcal{A}, X;\norm{\cdot}_{\cX}) = \E \norm{X-\tilde{X}}_{\cX}.
	\end{align*}
\end{lemma}
Let $\{X_j\}_{j \in J}$ be a set of $\mathcal{X}$-valued random variables with index set $J$ of finite cardinality. Then define
\begin{align*}
\tau(\mathcal{A}, \{X_j \in \mathcal{X}\}_{j \in J}; \norm{\cdot}_{\cX}) = \Big\lVert\sup_{g \in \Lambda(\norm{\cdot}_{\cX}')}\Big\{\int g(x)\P_{\{X_j\}_{j \in J}|\mathcal{A}}({\sf d}x) - \int g(x)\P_{\{X_j \}_{j \in J}}({\sf d}x) \Big\}\Big\lVert_{L(1)},
\end{align*}
where $\P_{\{X_j \}_{j \in J}}$ is the distribution of $\{X_j \}_{j \in J}$, $\P_{\{X_j\}_{j \in J}|\mathcal{A}}$ is the conditional distribution of $\{X_j \}_{j \in J}$ given $\mathcal{A}$, and $\Lambda(\norm{\cdot}_{\cX}')$ stands for the set of 1-Lipschitz functions from $\underbrace{\mathcal{X}\times \cdots \times \mathcal{X}}_{\card(J)}$ to $\R$ with respect to the norm $\norm{x}_{\cX}' := \sum_{j \in J}\norm{x_j}_{\cX}$ induced by $\norm{\cdot}_{\cX}$ for any $x=(x_1,\ldots,x_J)\in\cX^{{\card}(J)}$. 

Using these concepts, for a sequence of temporally dependent data $\{X_t\}_{t\in \Z}$, we are ready to define measure of temporal correlation strength as follows,
\[
\tau(k; \{X_t\}_{t \in \Z}, \norm{\cdot}_{\cX}) := \sup_{i > 0}\max_{1 \leq \ell \leq i} \frac{1}{\ell}\sup\{\tau\{\sigma(X_{-\infty}^a), \{X_{j_1}, \dots, X_{j_\ell}\}; \norm{\cdot}_{\cX}\}, a+k \leq j_1 < \dots < j_\ell\},
\]
where the inner supremum is taken over all $a \in \Z$ and all $\ell$-tuples $(j_1, \dots, j_\ell)$. $\{X_t\}_{t\in \Z}$ is said to be $\tau$-mixing if $\tau(k; \{X_t\}_{t \in \Z}, \norm{\cdot}_{\cX})$ converges to zero as $k\to\infty$. 
In \cite{dedecker2007book} the authors gave numerous examples of random sequences that are $\tau$-mixing.

\subsection{Overview of proof of Theorem \ref{thm:bern}}\label{subsec:bernover}
The proof of Theorem \ref{thm:bern} follows largely the proof of Theorem 1 in \cite{banna2016bernstein}. Section \ref{sec:cantor} reviews the Cantor-set construction developed and used in \cite{merlevede2009bernstein} and \cite{banna2016bernstein}. Lemma \ref{decoup} is a slight extension of Lemma 8 in \cite{banna2016bernstein}. The major difference is that the 0-1 function used to quantify the distance between two random matrices under $\beta$-mixing by Berbee's decoupling lemma \citep{berbee1979random} is replaced by an absolute distance function, which is used under $\tau$-mixing by Lemma \ref{lem:tauup} \citep{dedecker2004coupling}. Proofs of Lemma \ref{lem:bernlm2} and the rest of Theorem \ref{thm:bern} follow largely the proofs of Proposition 7 and Theorem 1 in \cite{banna2016bernstein} respectively, though with more algebras involved.

\subsection{Construction of Cantor-like set}\label{sec:cantor}
We follow \cite{banna2016bernstein} to construct the Cantor-like set $K_B$ for $\lbrace 1, \dots, B\rbrace$. Let $\delta = \frac{\log 2}{2\log B}$ and $\ell_B = \sup\lbrace k\in \Z^+: \frac{B\delta(1-\delta)^{k-1}}{2^k} \geq 2\rbrace$. We abbreviate $\ell := \ell_B$. Let $n_0 = B$ and for $j\in \lbrace 1,\dots, \ell\rbrace$,
\begin{align*}
n_j = \Big\lceil \frac{B(1-\delta)^j}{2^j}\Big\rceil~~{\rm and}~~d_{j-1} = n_{j-1} - 2n_j.
\end{align*}
We start from the set $\lbrace 1,\dots, B\rbrace$ and divide the set into three disjoint subsets $I_1^1, J_0^1, I_1^2$ so that $\card(I_1^1) = \card(I_1^2) = n_1$ and $\card(J_0^1) = d_0$. Specifically, 
$$I_1^1 = \{1, \dots, n_1\},\ J_0^1 = \{n_1+1, \dots, n_1+d_0\},\ I_1^2 = \{n_1 + d_0+1, \dots, 2n_1 + d_0\},$$ where $B = 2n_1 + d_0$. Then we divide $I_1^1, I_1^2$ with $J_0^1$ unchanged. $I_1^1$ is divided into three disjoint subsets $I_2^1, J_1^1, I_2^2$ in the same way as the previous step with $\card(I_2^1) = \card(I_2^2) = n_2$ and $\card(J_1^1) = d_1$. We obtain 
$$I_2^1 = 
\{1,\dots, n_2\},\ J_1^1 = \{n_2+1, \dots, n_2 + d_1\},\ I_2^2 = \{n_2 + d_1 + 1, \dots, 2n_2+d_1\},$$ where $n_1 = 2n_2+d_1$. Similarly, $I_1^2$ is divided into $I_2^3, J_1^2, I_2^4$ with $\card(I_2^3) = \card(I_2^4) = n_2$ and $\card(J_1^2) = d_1$. We obtain 
\begin{align*}
I_2^3 &= 
\{2n_2 + d_0 + d_1+1,\dots, 3n_2 + d_0 + d_1\},\ J_1^2 = \{3n_2 + d_0 + d_1+1, \dots, 3n_2 + d_0 + 2d_1\},\\
I_2^4 &= \{3n_2 + d_0 + 2d_1 + 1, \dots, 4n_2 + d_0 + 2d_1\},
\end{align*}
 where $B = 4n_2 + d_0 + 2d_1$. 
 
 Suppose we iterate this process for $k$ times ($k \in \lbrace 1,\dots, \ell\rbrace$) with intervals $I_k^i, i \in \lbrace 1,\dots, 2^k\rbrace$. For each $I_k^i$, we divide it into three disjoint subsets $I_{k+1}^{2i-1}, J_{k}^i, I_{k+1}^{2i}$ so that $\card(I_{k+1}^{2i-1}) = \card(I_{k+1}^{2i}) = n_{k+1}$ and $\card(J_{k}^i) = d_k$. More specifically, if $I_k^i = \{a_k^i, \dots, b_k^i\}$, then
 \begin{align*}
& I_{k+1}^{2i-1} = \{a_k^i, \dots, a_k^i + n_{k+1}-1\},\ J_k^i = \{a_k^i + n_{k+1}, \dots, a_k^i + n_{k+1} + d_k -1\},\\
 & I_{k+1}^{2i} = \{a_k^i + n_{k+1} + d_k, \dots, a_k^i + 2n_{k+1} + d_k-1\}.
 \end{align*}
 After $\ell$ steps, we obtain $2^\ell$ disjoint subsets $I_\ell^{i}, i\in \lbrace 1,\dots, 2^\ell\rbrace$ with $\card(I_\ell^{i}) = n_\ell$. Then the Cantor-like set is defined as
\begin{align*}
K_B = \bigcup\limits_{i = 1}^{2^\ell}I_{\ell}^i,
\end{align*}
and for each level $k \in \lbrace 0,\dots, \ell\rbrace$ and each $j \in \lbrace 1, \dots, 2^k\rbrace$, define
\begin{align*}
K_k^j = \bigcup\limits_{i = (j-1)2^{\ell-k}+1}^{j2^{\ell-k}}I_{\ell}^i.
\end{align*}
Some properties derived from this construction are given by \cite{banna2016bernstein}:
\begin{enumerate}
	\item $\delta \leq \frac{1}{2}$ and $\ell \leq \frac{\log B}{\log 2}$;
	\item $d_j \geq \frac{B\delta(1-\delta)^j}{2^{j+1}}$ and $n_\ell \leq \frac{B(1-\delta)^\ell}{2^{\ell-1}}$;
	\item Each $I_\ell^i, i\in \lbrace 1,\dots, 2^\ell\rbrace$ contains $n_\ell$ consecutive integers, and for any $i\in \lbrace 1,\dots, 2^{\ell-1}\rbrace$, $I_\ell^{2i-1}$ and $I_\ell^{2i}$ are spaced by $d_{\ell-1}$ integers;
	\item $\card(K_B) \geq \frac{B}{2}$; 
	\item For each $k \in \lbrace 0,\dots, \ell\rbrace$ and each $j \in \lbrace 1, \dots, 2^k\rbrace$, $\card(K_k^j) = 2^{\ell-k}n_\ell$. For each $j \in \lbrace 1, \dots, 2^{k-1}\rbrace$, $K_k^{2j-1}$ and $K_k^{2j}$ are spaced by $d_{k-1}$ integers;
	\item $K_0^1 = K_B$ and $K_\ell^j = I_\ell^j$ for $j \in \lbrace 1, \dots, 2^\ell\rbrace$.
\end{enumerate}

\subsection{A decoupling lemma for $\tau$-mixing random matrices}\label{sec:decoupling}
This section introduces the key tool to decouple $\tau$-mixing random matrices using Cantor-like set constructed in Section \ref{sec:cantor}. With some abuse of notation, within this section let's use $\lbrace \Xb_j \rbrace_{j\in\{1,\ldots,n\}}$ to denote a  generic sequence of $p\times p$ symmetric random matrices. Assume $\E(\Xb_j) = \mathbf{0}$ and $\lVert \Xb_j \rVert \leq M$ for some positive constant $M$ and for all $j\geq 1$. For a collection of index sets $H^k_1,\ k\in \lbrace 1,\dots, d\rbrace$, we assume that their cardinalities are equal and even. Denote $\lbrace\Xb_j \rbrace_{j \in H^k_1}$ to be the set of matrices whose indices are in $H^k_1$. Assume $\lbrace\Xb_j \rbrace_{j \in H_1^1}, \dots, \lbrace\Xb_j \rbrace_{j \in H^d_1}$ are mutually independent, while within each block $H_1^k$ the matrices are possibly dependent. For each $k$, decompose $H^k_1$ into two disjoint sets $H^{2k-1}_2$ and $H^{2k}_2$ with equal size, containing the first and second half of $H^k_1$ respectively. In addition, we denote $\tau_0:= \tau\{\sigma(\lbrace\Xb_j \rbrace_{j \in H^{2k-1}_2}),\ \lbrace\Xb_j \rbrace_{j \in H^{2k}_2};\norm{\cdot}\}$ for some constant $\tau_0 \geq 0$ and for all $k \in \lbrace 1,\dots, d\rbrace$. For a given $\epsilon > 0$, we achieve the following decoupling lemma.
\begin{lemma}\label{decoup}
	We obtain for any $\epsilon >0$,
	\begin{align*}
	&\E\tr\exp\Big(t\sum_{k = 1}^{d}\sum_{j \in H^k_1}\Xb_j\Big)
	\leq \sum\limits_{i = 0}^{d}\binom{d}{i}(1 + L_1 + L_2)^{d - i}(L_1)^{i}\E\tr\exp\Big\{(-1)^i t\Big(\sum\limits_{k = 1}^{2d}\sum_{j \in H^{k}_2} \tilde{\Xb}_j \Big)\Big\},\\
	&\E\tr\exp\Big(-t\sum_{k = 1}^{d}\sum_{j \in H^k_1}\Xb_j\Big)
	\leq \sum\limits_{i = 0}^{d}\binom{d}{i}(1 + L_1 + L_2)^{d - i}(L_1)^{i}\E\tr\exp\Big\{(-1)^{i+1} t\Big(\sum\limits_{k = 1}^{2d}\sum_{j \in H^k_2} \tilde{\Xb}_j \Big)\Big\},
	\end{align*}
	where 
	\[	
	L_1 := pt\epsilon\exp(t\epsilon), ~~L_2 := \exp\{\card(H_1^1)tM\}\tau_0/\epsilon, 
	\]	
	and $\lbrace\tilde{\Xb}_j \rbrace_{j \in H^k_2},\ k \in \lbrace 1, \dots, 2d\rbrace$, are mutually independent and have the same distributions as $\lbrace \Xb_j \rbrace_{j \in H^k_2}$, $k \in \lbrace 1, \dots, 2d\rbrace$.
\end{lemma}

\begin{proof}
	We prove this lemma by induction. For any $k\in \lbrace 1,\dots, d\rbrace$, we have $H^k_1 = H^{2k-1}_2 \cup H^{2k}_2$ and hence $\sum_{j \in H^k_1}\Xb_j = \sum_{j \in H^{2k-1}_2}\Xb_j + \sum_{j \in H^{2k}_2}\Xb_j$. 
	
	By Lemma \ref{lem:dedecker}, for each $k \in \lbrace 1,\dots, d\rbrace$, we could find a sequence of random matrices $\lbrace \tilde{\Xb}_j \rbrace_{j \in H^{2k}_2}$ and an independent uniformly distributed random variable $U_k$ on $[0,1]$ such that 
	\begin{enumerate}
		\item $\lbrace \tilde{\Xb}_j \rbrace_{j \in H^{2k}_2}$ is measurable with respect to the sigma field $\sigma(\lbrace \Xb_j \rbrace_{j \in H^{2k-1}_2})\vee \sigma(\lbrace \Xb_j \rbrace_{j \in H^{2k}_2}) \vee \sigma(U_k)$;
		\item $\lbrace \tilde{\Xb}_j \rbrace_{j \in H^{2k}_2}$ is independent of $\sigma(\lbrace \Xb_j \rbrace_{j \in H^{2k-1}_2})$;
		\item $\lbrace \tilde{\Xb}_j \rbrace_{j \in H^{2k}_2}$ has the same distribution as $\lbrace \Xb_j \rbrace_{j \in H^{2k}_2}$;
		\item $\P(\lVert \sum_{j \in H^{2k}_2} \Xb_j -  \sum_{j \in H^{2k}_2} \tilde{\Xb}_j \lVert > \epsilon_k) \leq \E(\lVert \sum_{j \in H^{2k}_2} \Xb_j -  \sum_{j \in H^{2k}_2} \tilde{\Xb}_j \lVert)/\epsilon_k \leq \tau_0/\epsilon_k$ by Markov's inequality and the fact that $\tau_0 = \sum_{j \in H^{2k}_2}\E(\lVert  \Xb_j -  \tilde{\Xb}_j \lVert)$.
	\end{enumerate}
	To make notation easier to follow, we set equal value to $\epsilon_k$ for $k \in \lbrace 1, \dots, d\rbrace$ and denote it as $\epsilon$. Moreover, we denote the event $\Gamma_{k} = \lbrace \lVert \sum_{j \in H^{2k}_2} \tilde{\Xb}_j -  \sum_{j \in H^{2k}_2} \Xb_j \lVert \leq \epsilon \rbrace$ for $k \in \lbrace 1, \dots, d\rbrace$.
	
	For the base case $k = 1$.
	\begin{align*}
	&\E\tr\exp\Big(t\sum_{k = 1}^{d}\sum_{j \in H^k_1}\Xb_j\Big)\! = \underbrace{\E\Big\{\mathds{1}_{\Gamma_1}\tr\exp\Big(t\sum_{k = 1}^{d}\sum_{j \in H^k_1}\Xb_j\Big)\Big\}\!}_{I}  + \underbrace{\! \E\Big\{\mathds{1}_{(\Gamma_1)^c}\tr\exp\Big(t\sum_{k = 1}^{d}\sum_{j \in H^k_1}\Xb_j\Big)\Big\}\!}_{II}.
	\end{align*}
Notice the definitions of terms $I$ and $II$ therein.
	
	We have
	\begin{align*}
	I &= \E\Big[\mathds{1}_{\Gamma_1}\tr\exp\Big\{t\Big(\sum_{j \in H^1_2}\Xb_j +\sum_{j \in H^2_2}\Xb_j + \sum_{k = 2}^{d}\sum_{j \in H^k_1}\Xb_j\Big)\Big\}\Big]\\
	&\leq \E\tr\exp\Big\{t\Big(\sum_{j \in H^1_2}\Xb_j +\sum_{j \in H^2_2}\tilde{\Xb}_j + \sum_{k = 2}^{d}\sum_{j \in H^k_1}\Xb_j\Big)\Big\} \\
	&+ \E\Big(\mathds{1}_{\Gamma_1}\Big[\tr\exp\Big\{t\Big(\sum_{j \in H^1_2}\Xb_j +\sum_{j \in H^2_2}\Xb_j + \sum_{k = 2}^{d}\sum_{j \in H^k_1}\Xb_j\Big)\Big\}\! -\tr\exp\Big\{t\Big(\sum_{j \in H^1_2}\Xb_j +\!\sum_{j \in H^2_2}\tilde{\Xb}_j + \sum_{k = 2}^{d}\sum_{j \in H^k_1}\Xb_j\Big)\Big\}\Big]\Big).
	\end{align*}
	By linearity of expectation and the facts that $\tr(\Xb) \leq p\lVert \Xb\rVert$ and $\lVert\exp(\Xb) - \exp(\Yb)\rVert \leq \lVert \Xb - \Yb\rVert\exp(\lVert \Xb - \Yb\rVert)\exp(\lVert \Yb\rVert)$, we obtain
	\begin{align*}
	&\E\Big(\mathds{1}_{\Gamma_1}\Big[\tr\exp\Big\{t\Big(\sum_{j \in H^1_2}\Xb_j +\sum_{j \in H^2_2}\Xb_j + \sum_{k = 2}^{d}\sum_{j \in H^k_1}\Xb_j\Big)\Big\}\! -\tr\exp\Big\{t\Big(\sum_{j \in H^1_2}\Xb_j +\!\sum_{j \in H^2_2}\tilde{\Xb}_j + \sum_{k = 2}^{d}\sum_{j \in H^k_1}\Xb_j\Big)\Big\}\Big]\Big)\\
	\leq& \E\Big[ \mathds{1}_{\Gamma_1}p\Big\lVert \exp\Big\{t\Big(\sum_{j \in H^1_2}\Xb_j +\sum_{j \in H^2_2}\Xb_j + \sum_{k = 2}^{d}\sum_{j \in H^k_1}\Xb_j\Big)\Big\} - \exp\Big\{t
	\Big(\sum_{j \in H^1_2}\Xb_j +\sum_{j \in H^2_2}\tilde{\Xb}_j + \sum_{k = 2}^{d}\sum_{j \in H^k_1}\Xb_j\Big)\Big\}\Big\lVert\Big]\\
	\leq& \E\Big[\mathds{1}_{\Gamma_1}p\Big\lVert t\sum_{j \in H^2_2}(\Xb_j - \tilde{\Xb}_j) \Big\rVert \exp\Big\{\Big\lVert t\sum_{j \in H^2_2}(\Xb_j - \tilde{\Xb}_j) \Big\rVert\Big\}\exp\Big\{\Big\lVert t\Big(\sum_{j \in H^1_2}\Xb_j +\sum_{j \in H^2_2}\tilde{\Xb}_j + \sum_{k = 2}^{d}\sum_{j \in H^k_1}\Xb_j\Big) \Big\rVert\Big\}\Big].
	\end{align*}
	
	By spectral mapping theorem, for a symmetric matrix $\Xb$ with $\norm{\Xb} \leq M$, we have $\exp(\lVert \Xb \rVert) \leq \lVert \exp(\Xb)\rVert \vee \lVert \exp(-\Xb)\rVert \leq \lVert \exp(\Xb)\rVert + \lVert \exp(-\Xb)\rVert $. Moreover, since $\exp(\Xb)$ is always positive definite for any matrix $\Xb$ and $\lVert \Xb \rVert \leq \tr(\Xb)$ for any positive definite symmetric matrix $\Xb$, we obtain $\lVert \exp(\Xb)\rVert \leq \tr\exp(\Xb)$ and $\lVert \exp(-\Xb)\rVert \leq \tr\exp(-\Xb)$. In addition, since we have $\norm{\sum_{j \in H_2^2}(\Xb_j - \tilde{\Xb}_j)} \leq \epsilon$ on $\Gamma_1$, we could further bound the inequality above by
	\begin{align*}
	&\E\Big[\mathds{1}_{\Gamma_1}pt\epsilon\exp(t\epsilon)\Big\lVert \exp\Big\{t\Big(\sum_{j \in H^1_2}\Xb_j +\sum_{j \in H^2_2}\tilde{\Xb}_j + \sum_{k = 2}^{d}\sum_{j \in H^k_1}\Xb_j\Big)\Big\}\Big\rVert\Big]\\
	\leq &pt\epsilon\exp(t\epsilon)
	\Big[\E\tr\exp \Big\{t\Big(\sum_{j \in H^1_2}\Xb_j +\sum_{j \in H^2_2}\tilde{\Xb}_j + \sum_{k = 2}^{d}\sum_{j \in H^k_1}\Xb_j\Big)\Big\} \\
	&+ \E\tr\exp \Big\{-t\Big(\sum_{j \in H^1_2}\Xb_j +\sum_{j \in H^2_2}\tilde{\Xb}_j + \sum_{k = 2}^{d}\sum_{j \in H^k_1}\Xb_j\Big)\Big\}\Big].
	\end{align*}
	Putting together, we reach
	\begin{align}\label{3.1.1}
	I \leq &\{1+pt\epsilon\exp(t\epsilon)\}\E\tr\exp \Big\{t\Big(\sum_{j \in H^1_2}\Xb_j +\sum_{j \in H^2_2}\tilde{\Xb}_j + \sum_{k = 2}^{d}\sum_{j \in H^k_1}\Xb_j\Big)\Big\}\nonumber\\
	&+ pt\epsilon\exp(t\epsilon)\E\tr\exp \Big\{-t\Big(\sum_{j \in H^1_2}\Xb_j +\sum_{j \in H^2_2}\tilde{\Xb}_j + \sum_{k = 2}^{d}\sum_{j \in H^k_1}\Xb_j\Big)\Big\}.
	\end{align}
	
	We then aim at $II$. For this, the proof largely follows the same argument as in \cite{banna2016bernstein}. Omitting the details, we obtain
	\begin{align}\label{3.1.2}
	II\leq& \exp\{\card(H^1_1)tM\}(\tau_0/\epsilon)\E\tr\exp\Big\{t\Big(\sum_{j \in H_1^2} \Xb_j + \sum_{j \in H_2^2} \tilde{\Xb}_j + \sum_{k = 2}^{d}\sum_{j \in H_k^1}\Xb_j\Big)\Big\}.
	\end{align}
	
	Denote $L_1 := pt\epsilon\exp(t\epsilon)$ and $L_2 := \exp\{\card(H_1^1)tM\}\tau_0/\epsilon$. Combining \eqref{3.1.1} and \eqref{3.1.2} yields
	\begin{align*}
	&\E\tr\exp\Big(t\sum_{k = 1}^{d}\sum_{j \in H^k_1}\Xb_j\Big)\\
	\leq &(1 + L_1 + L_2)\E\tr\exp\Big\{t\Big(\sum_{j \in H^1_2} \Xb_j + \sum_{j \in H^2_2} \tilde{\Xb}_j + \sum_{k = 2}^{d}\sum_{j \in H^k_1}\Xb_j\Big)\Big\}\\
	&+ L_1 \E\tr\exp\Big\{-t\Big(\sum_{j \in H^1_2} \Xb_j + \sum_{j \in H^2_2} \tilde{\Xb}_j + \sum_{k = 2}^{d}\sum_{j \in H^k_1}\Xb_j\Big)\Big\}\\
	=& \sum\limits_{i = 0}^{1}\binom{1}{i}(1 + L_1 + L_2)^{1 - i}(L_1)^{i}\E\tr\exp\Big\{(-1)^i t\Big(\sum_{j \in H^1_2} \Xb_j + \sum_{j \in H^2_2} \tilde{\Xb}_j + \sum_{k = 2}^{d}\sum_{j \in H^k_1}\Xb_j\Big)\Big\}.
	\end{align*}
	This finishes the base case. 
	
	The induction steps are followed similarly and we omit the details.
	By iterating $d$ times, we arrive at the following inequality:
	\begin{align}\label{3.1.3}
	&\E\tr\exp\Big(t\sum_{k = 1}^{d}\sum_{j \in H^k_1}\Xb_j\Big)\nonumber \\
	\leq& \sum\limits_{i = 0}^{d}\binom{d}{i}(1 + L_1 + L_2)^{d - i}(L_1)^{i}\E\tr\exp\Big\{(-1)^i t\Big(\sum\limits_{k = 1}^{d}\sum_{j \in H^{2k-1}_2} \Xb_j + \sum\limits_{k = 1}^{d}\sum_{j \in H^{2k}_2} \tilde{\Xb}_j\Big)\Big\},
	\end{align}
	where $\lbrace\Xb_j \rbrace_{j \in H^{2k-1}_2},\ k \in \lbrace 1, \dots, d\rbrace$ and $\lbrace\tilde{\Xb}_j \rbrace_{j \in H^{2k}_2},\ k \in \lbrace 1, \dots, d\rbrace$ are mutually independent. In addition, they have the same distributions as $\lbrace\Xb_j \rbrace_{j \in H^{2k-1}_2},\ k \in \lbrace 1, \dots, d\rbrace$ and $\lbrace \Xb_j \rbrace_{j \in H^{2k}_2},\ k \in \lbrace 1, \dots, d\rbrace$, respectively. For the sake of simplicity and clarity, we add an upper tilde to the matrices with indices in $H^{2k-1}_2,\ k \in \lbrace 1, \dots, d\rbrace$, i.e., $\lbrace\tilde{\Xb}_j \rbrace_{j \in H^{2k-1}_2}$ is identically distributed as $\lbrace\Xb_j \rbrace_{j \in H^{2k-1}_2}$ for $k \in \lbrace 1, \dots, d\rbrace$. Hence \eqref{3.1.3} could be rewritten as
	\begin{align*}
	&\E\tr\exp\Big(t\sum_{k = 1}^{d}\sum_{j \in H^k_1}\Xb_j\Big)
	\leq \sum\limits_{i = 0}^{d}\binom{d}{i}(1 + L_1 + L_2)^{d - i}(L_1)^{i}\E\tr\exp\Big\{(-1)^i t\Big(\sum\limits_{k = 1}^{2d}\sum_{j \in H^{k}_2} \tilde{\Xb}_j \Big)\Big\},
	\end{align*}
	where $\lbrace\tilde{\Xb}_j \rbrace_{j \in H^k_2},\ k \in \lbrace 1, \dots, 2d\rbrace$ are mutually independent and their distributions are the same as $\lbrace\Xb_j \rbrace_{j \in H^k_2},\ k \in \lbrace 1, \dots, 2d\rbrace$. 
	
	By changing $\Xb$ to $-\Xb$, we immediately get the following bound:
	\begin{align*}
	&\E\tr\exp\Big(-t\sum_{k = 1}^{d}\sum_{j \in H^k_1}\Xb_j\Big)
	\leq \sum\limits_{i = 0}^{d}\binom{d}{i}(1 + L_1 + L_2)^{d - i}(L_1)^{i}\E\tr\exp\Big\{(-1)^{i+1} t\Big(\sum\limits_{k = 1}^{2d}\sum_{j \in H^{k}_2} \tilde{\Xb}_j \Big)\Big\}.
	\end{align*}
	This completes the proof of Lemma \ref{decoup}.
\end{proof}

\subsection{Proof of Theorem \ref{thm:bern} }\label{sec:thm-main}
\begin{proof}
	Without loss of generality, let $\psi_1 = \tilde{\psi}_1$. 
	
	{\bf Case I.} First of all, we consider $M = 1$.
	
	{\bf Step I (Summation decomposition).} Let $B_0 = n$ and $\Ub^{(0)}_{j} = \Xb_{j}$ for $j \in \lbrace 1,\dots, n\rbrace$. Let $K_{B_0}$ be the Cantor-like set from $\lbrace 1,\dots, B_0\rbrace$ by construction of Section \ref{sec:cantor}, $K_{B_0}^c = \lbrace 1,\dots, B_0\rbrace \setminus K_{B_0}$, and $B_1 = \card(K_{B_0}^c)$. Then define
	\begin{align*}
	&\Ub^{(1)}_{j} = \Xb_{i_j}, \text{ where } i_j \in K_{B_0}^c = \lbrace i_1, \dots, i_{B_1}\rbrace.
	\end{align*}
	For each $i\geq 1$, let $K_{B_i}$ be constructed from $\lbrace 1,\dots, B_{i} \rbrace$ by the same Cantor-like set construction. Denote $K_{B_{i}}^c = \lbrace 1,\dots, B_{i} \rbrace \setminus K_{B_{i}}$ and $B_{i+1} = \card(K_{B_{i}}^c)$. Then
	\begin{align*}
	\Ub^{(i+1)}_{j} = \Ub^{(i)}_{k_j}, \text{ where } k_j \in K_{B_{i}}^c = \lbrace k_1, \dots, k_{B_{i+1}}\rbrace.
	\end{align*}
	We stop the process when there is a smallest $L$ such that $B_L \leq 2$. Then we have for $i\leq L-1$, $B_i \leq n2^{-i}$ because each Cantor-like set $K_{B_{i+1}}$ has cardinality greater than $B_{i}/2$. Also notice that $L \leq [\log n/\log 2]$.
	
	For $i \in \lbrace 0,\dots, L-1\rbrace$, denote
	\begin{align*}
	\Sb_i = \sum\limits_{j \in K_{B_i}} \Ub_j^{(i)} \text{ and } \Sb_L = \sum_{j \in K_{B_{L-1}^c}} \Ub^{(L)}_j.
	\end{align*}
	Then we observe
	\begin{align*}
	\sum_{j = 1}^n \Xb_j= \sum_{i = 0}^{L}\Sb_i.
	\end{align*}
	
	{\bf Step II (Bounding Laplacian transform).} This step hinges on the following lemma, which provides an upper bound for the Laplace transform of sum of a sequence of random matrices which are $\tau$-mixing with geometric decay, i.e.,~$\tau(k) \leq \psi_1\exp\{-\psi_2(k-1)\}$ for all $k \geq 1$ for some constants $\psi_1, \psi_2 > 0$.
	\begin{lemma}[Proof in Section \ref{sec:lem2}]\label{lem:bernlm2}
		For a sequence of $p \times p$ matrices $\lbrace \Xb_i \rbrace$, $i \in \lbrace 1,\dots, B\rbrace$ satisfying conditions in Theorem \ref{thm:bern} with $M = 1$ and $\psi_1\geq p^{-1}$, there exists a subset $K_B \subseteq \lbrace 1,\dots, B\rbrace$ such that for $0 < t \leq \min\{1, \frac{\psi_2}{8\log(\psi_1B^6p)}\}$,
		\begin{align*}
		\log\E\tr\exp\bigg(t\sum_{j \in K_B}\Xb_j\bigg)
		\leq \log p + 4h(4)Bt^2\nu^2 + 151\Big[1 + \exp\Big\{ \frac{1}{\sqrt{p}} \exp\Big(-\frac{\psi_2}{64t}\Big)\Big\}\Big]\frac{t^2}{\psi_2} \exp\Big(-\frac{\psi_2}{64t}\Big).
		\end{align*}
	\end{lemma}
	
	For each $\Sb_i, i \in \lbrace 0,\dots, L-1\rbrace$, by applying Lemma \ref{lem:bernlm2}  with $B = B_i$, we have for any positive $t$ satisfying $0 < t \leq  \min\{1,\frac{\psi_2}{8\log\{\psi_1 (n2^{-i})^6p\}}\}$,
	\begin{align*}
	&\log\E\tr\exp(t\Sb_i) 
	\leq \log p + t^2(C_12^{-i}n + C_{2,i})
	\end{align*}
	where $C_1 := 4h(4)\nu^2,C_{2,i} := 302\cdot2^{\frac{6i}{8}}/\psi_2n^{\frac{6}{8}}$.
	
	Denote
	\begin{align*}
	\tilde{f}(\psi_1, \psi_2, i) &:= \min\Big\{1, \frac{\psi_2}{8\log\{\psi_1 (n2^{-i})^6p\}}\Big\}.
	\end{align*}
	For any $0 < t \leq \tilde{f}(\psi_1, \psi_2, i)$, we obtain
	\begin{align*}
	\log\E\tr\exp(t\Sb_i)
	\leq \log p + \frac{t^2(C_12^{-i}n + C_{2,i})}{1 - t/\tilde{f}(\psi_1, \psi_2, i)} \leq  \log p + \frac{t^2\{C_1^{\frac{1}{2}}(2^{-i}n)^{\frac{1}{2}} + C_{2,i}^{\frac{1}{2}}\}^2}{1 - t/\tilde{f}(\psi_1, \psi_2, i)}.
	\end{align*}
	
	For $\Sb_L$, since $B_L \leq 2$, for $0 < t \leq 1$,
	\begin{align*}
	\log\E\tr\exp(t\Sb_L) &\leq \log p + t^2h(2t)\lambda_{\max}\{\E(\Sb_L^2)\} \leq \log p + \frac{2t^2\nu^2}{1-t}. 
	\end{align*}
	
	Denote
	$\sigma_i := C_1^{\frac{1}{2}}(2^{-i}n)^{\frac{1}{2}} + C_{2,i}^{\frac{1}{2}},\ \sigma_L := \sqrt{2}\nu,\ \kappa_i := 1/\tilde{f}(\psi_1, \psi_2, i),\ \text{and }\ \kappa_L := 1$.
	
	Summing up, we have
	\begin{align*}
	\sum_{i = 0}^{L} \sigma_i &= \sum_{i = 0}^{L-1} \{C_1^{\frac{1}{2}}(2^{-i}n)^{\frac{1}{2}} + C_{2,i}^{\frac{1}{2}} \}+ \sqrt{2}\nu \leq  15\sqrt{n}\nu + 60\sqrt{1/\psi_2},\\
	\sum_{i = 0}^L \kappa_i &\leq \frac{\log n}{\log 2}\max\Big\{1,\frac{8\log(\psi_1n^6p)}{\psi_2}\Big\}:= \tilde{\psi}(\psi_1,\psi_2,n,p).
	\end{align*}
	
	Hence by Lemma 3 in \cite{merlevede2009bernstein}, for $0 < t \leq \{\tilde{\psi}(\psi_1,\psi_2,n,p)\}^{-1}$, we have 
	\begin{align*}
	\log\E\tr\exp\Big(t\sum_{j = 1}^n\Xb_j\Big)
	\leq\log p + \frac{t^2\Big(15\sqrt{n}\nu +60\sqrt{1/\psi_2}\Big)^2}{1 - t\tilde{\psi}(\psi_1,\psi_2,n,p)}.
	\end{align*}
	
	{\bf Step III (Matrix Chernoff bound).} Lastly by matrix Chernoff bound, we obtain
	\begin{align*}
	\P\Big\{\lambda_{\max}\Big(\sum_{j = 1}^{n}\Xb_j\Big) \geq x\Big\}
	\leq p \exp\Big\{-\frac{x^2}{8(15^2n\nu^2 + 60^2/\psi_2) + 2x\tilde{\psi}(\psi_1,\psi_2,n,p)}\Big\}.
	\end{align*}
	
	{\bf Case II.} We consider general $M>0$. It is obvious that if $\{\Xb_t\}_{t \in \Z}$ is a sequence of $\tau$-mixing random matrices such that $\tau(k; \{\Xb_t\}_{t \in \Z}, \norm{\cdot}) \leq M\psi_1\exp\{-\psi_2(k-1)\}$, then $\{\Xb_i/M\}_{i \in \Z}$ is also a sequence of $\tau$-mixing random matrices such that $\tau(k; \{\Xb_t/M\}_{t \in \Z}, \norm{\cdot}) \leq \psi_1\exp\{-\psi_2(k-1)\}$ and $\norm{\Xb_t/M} \leq 1$. Then applying the result of Case I to $\{\Xb_i/M\}_{i \in \Z}$, we obtain
	\begin{align*}
	&\P\Big\{\lambda_{\max}\Big(\sum_{j = 1}^{n}\Xb_j/M\Big) \geq x\Big\} 
	\leq  p \exp\Big\{-\frac{x^2}{8(15^2n\nu_M^2 + 60^2/\psi_2) + 2x\tilde{\psi}(\psi_1,\psi_2,n,p)}\Big\},
	\end{align*}
	where $\nu_M^2 := \sup_{K\subseteq \lbrace 1,\dots, n \rbrace}\frac{1}{\card(K)}\lambda_{\max}\bigg\{\E\bigg(\sum_{i \in K}\Xb_i/M\bigg)^2\bigg\} = \nu^2/M^2 $ for $\nu^2$ defined in Theorem \ref{thm:bern}. Thus
	\begin{align*}
	&\P\Big\{\lambda_{\max}\Big(\sum_{j = 1}^{n}\Xb_j\Big) \geq x\Big\} 
	\leq  p \exp\Big\{-\frac{x^2}{8(15^2n\nu^2 + 60^2M^2/\psi_2) + 2xM\tilde{\psi}(\psi_1,\psi_2,n,p)}\Big\}.
	\end{align*}
	This completes the proof of Theorem \ref{thm:bern}.
\end{proof}

\subsection{The proof of Lemma \ref{lem:bernlm2}}\label{sec:lem2}
\begin{proof}
	Let $K_B$ be constructed as in Section \ref{sec:cantor} for any arbitrary $B \geq 2$ and $M = 1$.
	
	{\bf Case I.} If $0 < t \leq 4/B$, by Lemma 4 in \cite{banna2016bernstein}, we have
	\begin{align*}
	\E\tr\exp\Big(t\sum_{i\in K_B} \Xb_i\Big) \leq p\exp\Big[t^2h\Big\{t\lambda_{\max}\Big(\sum_{i\in K_B} \Xb_i\Big)\Big\}
	\lambda_{\max}\Big\{\E\Big(\sum_{i\in K_B}\Xb_i\Big)^2\Big\}\Big]. 
	\end{align*}
	By Weyl's inequality, $\lambda_{\max}(\sum_{i\in K_B} \Xb_i) \leq B$ since $\card(K_B) \leq B$, and by definition of $\nu^2$ in Theorem \ref{thm:bern}, we have $\lambda_{\max}\{\E(\sum_{i\in K_B}\Xb_i)^2\} \leq B\nu^2$. Therefore, we obtain $h\{t\lambda_{\max}(\sum_{i\in K_B} \Xb_i)\} \leq h(tB) \leq h(4)$ and 
	\begin{align}\label{3.2.0}
	\E\tr\exp\Big(t\sum_{i\in K_B} \Xb_i\Big) \leq p\exp\{t^2h(4)B\nu^2\}. 
	\end{align}
	
	{\bf Case II.} Now we consider the case where $4/B < t \leq   \min\{1, \frac{\psi_2}{8\log(\psi_1B^6p)}\}$.
	
	{\bf Step I.}  Let $J$ be a chosen integer from $\lbrace 0, \dots, \ell_B\rbrace$ whose actual value will be determined later. We will use the same notation to denote Cantor-like sets as in Section \ref{sec:cantor}. By Lemma \ref{decoup} and similar induction argument as in \cite{banna2016bernstein}, we obtain
	\begin{align}\label{2.2.3}
	\E\tr\exp\Big(t\sum_{j \in K_0^1}\Xb_j\Big)
	\leq &\sum\limits_{i_1 = 0}^{2^0}\dots\sum\limits_{i_J = 0}^{2^{J-1}}\Big[\Big(\prod\limits_{k = 1}^{J} A_{k,i_k}\Big)\E\tr\exp\Big\{(-1)^{\sum_{k = 1}^{J} i_k}t\Big(\sum\limits_{i' = 1}^{2^J}\sum_{j \in K_{J}^{i'}}\tilde{\Xb}_j\Big)\Big\}\Big],
	\end{align}
	where $\lbrace\tilde{\Xb}_j\rbrace_{j \in K_{J}^{i'}}$ for $i' \in \lbrace 1, \cdots, 2^J \rbrace$ are mutually independent and have the same distributions as $\lbrace\Xb_j\rbrace_{j \in K_{J}^{i'}}$ for $i' \in \lbrace 1, \cdots, 2^J \rbrace$ ,and 
	\begin{align*}
	A_{k,i_k} := & \binom{2^{k-1}}{i_k}(1 + L_{k,1} + L_{k,2})^{2^{k-1}-i_k}(L_{k,1})^{i_k},\\
	\epsilon_k 
	:=&(2pt)^{-\frac{1}{2}}\{2^{\ell-k}n_\ell\exp(t2^{\ell-k+1} n_\ell)\tau_{d_{k-1}+1}\}^{\frac{1}{2}},\\
	L_{k,1} := &(pt/2)^{\frac{1}{2}}\exp(t\epsilon_k)\{2^{\ell-k}n_\ell\exp(t2^{\ell-k+1} n_\ell)\tau_{d_{k-1}+1}\}^{\frac{1}{2}},\\
	L_{k,2} := &(2pt)^{\frac{1}{2}}\exp(t\epsilon_k)\{2^{\ell-k}n_\ell\exp(t2^{\ell-k+1} n_\ell)\tau_{d_{k-1}+1}\}^{\frac{1}{2}},
	\end{align*}

	\textbf{Step II:} Now we choose $J$ as follows:
	\begin{align*}
	J & = \inf \Big\lbrace k\in \lbrace 0, \dots, \ell\rbrace: \frac{B(1-\delta)^k}{2^k} \leq \min \Big\{\frac{\psi_2}{8t^2}, B\Big\}\Big\rbrace.
	\end{align*}
	
	We first bound $\E\tr\exp\{t(\sum\limits_{i' = 1}^{2^J}\sum_{j \in K_{J}^{i'}}\tilde{\Xb}_j)\}$ and $\E\tr\exp\{-t(\sum\limits_{i' = 1}^{2^J}\sum_{j \in K_{J}^{i'}}\tilde{\Xb}_j)\}$. From \eqref{2.2.3} we obtain $2^J$ sets of $\lbrace \tilde{\Xb}_j\rbrace$ that are mutually independent. To make notation less cluttered, we will remove the upper tilde from $ \tilde{\Xb}_j$ for all $j$. Denote the number of matrices in each set $K_{J}^i$ to be $q := 2^{\ell-J}n_\ell$. For each set $K_{J}^i,\ i \in \lbrace 1,\dots, 2^J\rbrace$, we divide it into consecutive sets with cardinality $\tilde{q}$ and potentially a residual term if $q$ is not divisible by $\tilde{q}$. More specifically, we have $2\tilde{q} \leq q$ and $m_{q,\tilde{q}} := [q/2\tilde{q}]$. The value $\tilde q$ will be determined later.
	
	Then each set $K_{J}^i$ contains $2m_{q,\tilde{q}}$ numbers of sets with cardinality $\tilde{q}$ and one set with cardinality less than $2\tilde{q}$. For each $K_{J}^i,\ i \in \lbrace 1, \dots, 2^J\rbrace$, denote these consecutive sets described above by $Q_k^i,\ k \in \lbrace 1,\dots, 2m_{q,\tilde{q}}+1\rbrace$. Given these notation, we could rewrite the bound as the following:
	\begin{align*}
	&\E\tr\exp\Big(t\sum_{i = 1}^{2^J}\sum_{j \in K_{J}^i}\Xb_j\Big)\\
	= &\E\tr\exp\Big(t\sum_{i = 1}^{2^J}\sum_{k = 1}^{2m_{q,\tilde{q}}+1}\sum_{j \in Q_k^i}\Xb_j\Big)
	=\E\tr\exp\Big(t\sum_{i = 1}^{2^J}\sum_{k = 1}^{m_{q,\tilde{q}}}\sum_{j \in Q_{2k}^i}\Xb_j + t\sum_{i = 1}^{2^J}\sum_{k = 1}^{m_{q,\tilde{q}}+1}\sum_{j \in Q_{2k-1}^i}\Xb_j\Big).
	\end{align*}
	Since $\tr\exp(\cdot)$ is convex (cf. Proposition 2 in \cite{petz1994survey}), by Jensen's inequality, we have 
	\begin{align*}
	\E\tr\exp\Big(t\sum_{i = 1}^{2^J}\sum_{j \in K_{J}^i}\Xb_j\Big) &\leq \frac{1}{2}\E\tr\exp\Big(2t\sum_{i = 1}^{2^J}\sum_{k = 1}^{m_{q,\tilde{q}}}\sum_{j \in Q_{2k}^i}\Xb_j\Big) + \frac{1}{2}\E\tr\exp\Big(2t\sum_{i = 1}^{2^J}\sum_{k = 1}^{m_{q,\tilde{q}}+1}\sum_{j \in Q_{2k-1}^i}\Xb_j\Big).
	\end{align*}
	Since the number of odd index sets is always equal to or one more than that of the even index sets, the upper bound of $\frac{1}{2}\E\tr\exp\Big(2t\sum_{i = 1}^{2^J}\sum_{k = 1}^{m_{q,\tilde{q}}}\sum_{j \in Q_{2k}^i}\Xb_j\Big)$ will always be less than or equal to that of $\frac{1}{2}\E\tr\exp\Big(2t\sum_{i = 1}^{2^J}\sum_{k = 1}^{m_{q,\tilde{q}}+1}\sum_{j \in Q_{2k-1}^i}\Xb_j\Big)$. Hence we only need to provide an upper bound for 
	$\E\tr\exp\Big(2t\sum_{i = 1}^{2^J}\sum_{k = 1}^{m_{q,\tilde{q}}+1}\sum_{j \in Q_{2k-1}^i}\Xb_j\Big)$. Our goal is then to replace all $\lbrace\Xb_j\rbrace_{j\in Q_{2k-1}^i}$ in the last inequality by mutually independent copies $\lbrace \tilde{\Xb}_j\rbrace_{j\in Q_{2k-1}^i}$ with same distributions for $k \in \lbrace 1,\dots, 2m_{q,\tilde{q}}+1 \rbrace,\ i \in \lbrace 1,\dots, 2^J\rbrace$. Again we will proceed by induction. We first show
	\begin{align*}
	&\E\tr\exp\Big(2t\sum_{i = 1}^{2^J}\sum_{k = 1}^{m_{q,\tilde{q}}+1}\sum_{j \in Q_{2k-1}^i}\Xb_j\Big)\\ 
	\leq & 
	\sum\limits_{i_1 = 0}^{1}\tilde{A}_{i_1}
	\E\tr\exp\Big\{(-1)^{i_1}2t\Big(\sum_{k = 1}^{m_{q,\tilde{q}}+1}\sum_{j \in Q_{2k-1}^1}\tilde{\Xb}_j + \sum_{i = 2}^{2^J}\sum_{k = 1}^{m_{q,\tilde{q}}+1}\sum_{j \in Q_{2k-1}^i}\Xb_j\Big)\Big\},
	\end{align*}
	where the constants $\tilde A_{i_1}$ will be specified later. For each $\lbrace \Xb_j\rbrace_{j \in Q_{2k-1}^1},\ k \in \lbrace 1, \dots, m_{q,\tilde{q}}+1\rbrace$, we could find a sequence of $\lbrace \tilde{\Xb}_j\rbrace_{j \in Q_{2k-1}^1},\  k \in \lbrace 1, \dots, m_{q,\tilde{q}}+1\rbrace$ that are mutually independent with each other. More specifically, let $\lbrace \tilde{\Xb}_j\rbrace_{j\in Q_1^1}  = \lbrace \Xb_j \rbrace_{j\in Q_1^1}$. By applying Lemma \ref{lem:dedecker} on $\lbrace \tilde{\Xb}_j\rbrace_{j\in Q_1^1}$ and $\lbrace \Xb_j\rbrace_{j\in Q_3^1}$ with a chosen $\tilde{\epsilon} > 0$, we may find a sequence of random matrices $\lbrace \tilde{\Xb}_j\rbrace_{j \in Q_3^1}$ such that for each $j_0\in Q_3^1$, we have
	\begin{enumerate}
		\item $\tilde{\Xb}_{j_0}$ is measurable with respect to $\sigma(\lbrace \tilde{\Xb}_j\rbrace_{j\in Q_1^1}) \vee \sigma(\Xb_{j_0}) \vee \sigma(\tilde{U}_{j_0}^1)$;
		\item $\tilde{\Xb}_{j_0}$ is independent of $\sigma(\lbrace \tilde{\Xb}_j\rbrace_{j\in Q_1^1})$;
		\item $\tilde{\Xb}_{j_0}$ has the same distribution as $\Xb_{j_0}$;
		\item $\P(\parallel \tilde{\Xb}_{j_0} - \Xb_{j_0} \parallel \geq \tilde{\epsilon}) \leq \E(\parallel \tilde{\Xb}_{j_0} - \Xb_{j_0} \parallel)/\tilde{\epsilon} \leq \tau_{\tilde{q}+1}/\tilde{\epsilon}$ by Markov's inequality.
	\end{enumerate}
	For each $j_0\in Q_3^1$, $\tilde{U}_{j_0}^1$ is independent with $\lbrace \tilde{\Xb}_j\rbrace_{j\in Q_1^1}$ and $\Xb_{j_0}$. In addition, since there are at least $\tilde{q}$ number of matrices between $\lbrace \tilde{\Xb}_j\rbrace_{j\in Q_1^1}$ and $\Xb_{j_0}$ by our construction, we have $\tau\{\sigma(\lbrace \tilde{\Xb}_j\rbrace_{j\in Q_1^1}), \Xb_{j_0};       \norm{\cdot}\} \leq \tau_{\tilde{q}+1}$. Note that $\lbrace \tilde{\Xb}_j\rbrace_{j \in Q_3^1}$ is independent with $\lbrace \tilde{\Xb}_j \rbrace_{j\in Q_1^1}$ but not mutually independent within the set $Q_3^1$. 
	
	Following the induction steps similar to the previous step and without redundancy, we obtain
	\begin{align*}
	&\E\tr\exp\Big(2t\sum_{i = 1}^{2^J}\sum_{k = 1}^{m_{q,\tilde{q}}+1}\sum_{j \in Q_{2k-1}^i}\Xb_j\Big)\\ 
	\leq & 
	\sum\limits_{i_1 = 0}^{1}\tilde{A}_{i_1}
	\E\tr\exp\Big\{(-1)^{i_1}2t\Big(\sum_{k = 1}^{m_{q,\tilde{q}}+1}\sum_{j \in Q_{2k-1}^1}\tilde{\Xb}_j + \sum_{i = 2}^{2^J}\sum_{k = 1}^{m_{q,\tilde{q}}+1}\sum_{j \in Q_{2k-1}^i}\Xb_j\Big)\Big\},
	\end{align*}
	where 
	\begin{align*}
	\tilde{\epsilon} &:= (4pt)^{-\frac{1}{2}}\{\exp(2tq)\tau_{\tilde{q}+1}\}^{\frac{1}{2}},\\
	\tilde{L}_{1} &:= \frac{1}{2}(4pt)^{\frac{1}{2}}q\exp(2tq\tilde{\epsilon})\{\exp(2tq)\tau_{\tilde{q}+1}\}^{\frac{1}{2}},\\
	\tilde{L}_{2} &:= (4pt)^{\frac{1}{2}}q\{\exp(2tq)\tau_{\tilde{q}+1}\}^{\frac{1}{2}},\\
	\tilde{A}_{i_1} &:= \binom{1}{i_1}(1 + \tilde{L}_{1} + \tilde{L}_{2})^{1 - i_1}(\tilde{L}_{1})^{i_1},
	\end{align*}
	This completes the base case. 
	
	Iterating the above calculation, we arrive at the following bound:
	\begin{align}\label{2.3.2}
	&\E\tr\exp\Big(2t\sum_{i = 1}^{2^J}\sum_{k = 1}^{m_{q,\tilde{q}}+1}\sum_{j \in Q_{2k-1}^i}\Xb_j\Big) \nonumber\\
	\leq &\sum\limits_{i_1 = 0}^{1}\dots \sum\limits_{i_{2^J} = 0}^{1}\Big(\prod\limits_{r = 1}^{2^J}\! \tilde{A}_{i_r}\Big)
	\E\tr\exp\Big\{(-1)^{\sum_{r = 1}^{2^J}i_{r}}2t\sum_{i = 1}^{2^J}\sum_{k = 1}^{m_{q,\tilde{q}}+1}\sum_{j \in Q_{2k-1}^i}\tilde{\Xb}_j\Big\},
	\end{align}
	where $\lbrace \tilde{\Xb}_j\rbrace_{j \in Q_{2k-1}^i}$ for $(i,k) \in \lbrace 1,\dots, 2^J \rbrace \times \lbrace 1, \dots, m_{q,\tilde{q}}+1 \rbrace$ are mutually independent and identically distributed as  $\lbrace \Xb_j\rbrace_{j \in Q_{2k-1}^i}$ for $(i,k) \in \lbrace 1,\dots, 2^J \rbrace \times \lbrace 1, \dots, m_{q,\tilde{q}}+1 \rbrace$, and
	\begin{align*}
	\tilde{\epsilon} &:= (4pt)^{-\frac{1}{2}}\{\exp(2tq)\tau_{\tilde{q}+1}\}^{\frac{1}{2}},\\
	\tilde{L}_{1} &:= \frac{1}{2}(4pt)^{\frac{1}{2}}q\exp(2tq\tilde{\epsilon})\{\exp(2tq)\tau_{\tilde{q}+1}\}^{\frac{1}{2}},\\
	\tilde{L}_{2} &:= (4pt)^{\frac{1}{2}}q\{\exp(2tq)\tau_{\tilde{q}+1}\}^{\frac{1}{2}},\\
	\tilde{A}_{i_r} &:= \binom{1}{i_r}(1 + \tilde{L}_{1} + \tilde{L}_{2})^{1 - i_r}(\tilde{L}_{1})^{i_r}.
	\end{align*}
	
	Let $\tilde{q} := [2/t]\wedge [q/2]$. $\lbrace \tilde{\Xb}_j \rbrace_{j \in Q_{2k-1}^i}$ for $(i,k) \in \lbrace 1, \dots, 2^J\rbrace\times \lbrace 1,\dots, m_{q,\tilde{q}} +1 \rbrace$ are mutually independent with mean $\mathbf{0}$ and $2^J\sum_{k = 1}^{m_{\tilde{q}, q} + 1}\card(Q_{2k-1}^i) \leq B$. Moreover by Weyl's inequality, for $(i,k) \in \lbrace 1, \dots, 2^J\rbrace\times \lbrace 1,\dots, m_{q,\tilde{q}} +1 \rbrace$, we have
	\begin{align*}
	2\lambda_{\max}\Big(\sum_{j \in Q_{2k-1}^i} \tilde{\Xb}_j\Big) \leq 2\tilde{q} \leq \frac{4}{t}.
	\end{align*}
	By Lemma 4 in \cite{banna2016bernstein}, we obtain 
	\begin{align}
	&\E\tr\exp\Big(2t\sum_{i = 1}^{2^J}\sum_{k = 1}^{m_{q,\tilde{q}}+1}\sum_{j \in Q_{2k-1}^i}\tilde{\Xb}_j\Big) \leq p\exp\{4h(4)Bt^2\nu^2\},\label{2.3.2'}\\
	&\E\tr\exp\Big(-2t\sum_{i = 1}^{2^J}\sum_{k = 1}^{m_{q,\tilde{q}}+1}\sum_{j \in Q_{2k-1}^i}\tilde{\Xb}_j\Big) \leq p\exp\{4h(4)Bt^2\nu^2\}.\label{2.3.2''}
	\end{align}
	
	Plugging \eqref{2.3.2'} and \eqref{2.3.2''} into \eqref{2.3.2} and using the fact that $\sum\limits_{i_r = 0}^{1} \tilde{A}_{i_r} = 1 + 2\tilde{L}_1 + \tilde{L}_2$, we obtain
	\begin{align}\label{2.3.2'''}
	&\E\tr\exp\Big(2t\sum_{i = 1}^{2^J}\sum_{k = 1}^{m_{q,\tilde{q}}+1}\sum_{j \in Q_{2k-1}^i}\Xb_j\Big)
	\leq (1 + 2\tilde{L}_1 + \tilde{L}_2)^{2^J}
	p\exp\{4h(4)Bt^2\nu^2\}.
	\end{align}
	
	By replacing $\Xb$ by $-\Xb$, we obtain
	\begin{align}\label{2.3.2''''}
	&\E\tr\exp\Big(-2t\sum_{i = 1}^{2^J}\sum_{k = 1}^{m_{q,\tilde{q}}+1}\sum_{j \in Q_{2k-1}^i}\Xb_j\Big)
	\leq (1 + 2\tilde{L}_1 + \tilde{L}_2)^{2^J}
	p\exp\{4h(4)Bt^2\nu^2\}.
	\end{align}
	
	Combining \eqref{2.2.3} with \eqref{2.3.2'''} and \eqref{2.3.2''''}, we get
	\begin{align}\label{2.3.3}
	\E\tr\exp\Big(t\sum_{j \in K_B}\Xb_j\Big)
	\leq &
	\sum\limits_{i_1 = 0}^{2^0}\dots\sum\limits_{i_{J} = 0}^{2^{J-1}}\Big[ \Big( \prod\limits_{k = 1}^{J}A_{k,i_k}\Big) (1 + 2\tilde{L}_1 + \tilde{L}_2)^{2^J}
	p\exp\{4h(4)Bt^2\nu^2\}\Big]\nonumber\\
	= & \Big\{ \prod\limits_{k = 1}^{J} (1 + 2L_{k,1} + L_{k,2})^{2^{k-1}}\Big\}(1 + 2\tilde{L}_1 + \tilde{L}_2)^{2^J}
	p\exp\{4h(4)Bt^2\nu^2\},
	\end{align}
	where the last equality follows by $\sum_{i_k = 1}^{2^{k-1}}A_{k,i_k} = (1 + 2L_{k,1} + L_{k,2})^{2^{k-1}}$.
	
	By using $\log(1 + x) \leq x$ for $x \geq 0$, we have
	\begin{align}\label{2.3.5}
	&\log\E\tr\exp\Big(t\sum_{j \in K_B}\Xb_j\Big)
	\leq \sum\limits_{k = 1}^{J} 2^{k -1}(2L_{k,1} + L_{k,2})
	+2^J(2\tilde{L}_1 + \tilde{L}_2)
	+\log[p\exp\{4h(4)Bt^2\nu^2\}].
	\end{align}
	For simplicity, we denote $I = \sum\limits_{k = 1}^{J} 2^{k -1}(2L_{k,1} + L_{k,2}),\ II = 2^J(2\tilde{L}_1 + \tilde{L}_2)$ in \eqref{2.3.5}.
	
	\textbf{Step III:} Following calculations similar to \cite{banna2016bernstein}, we obtain
	\begin{align}\label{2.3.7}
	I
	&\leq \frac{32\sqrt{2}}{\log 2}\Big[1 + \exp\Big\{\frac{1}{\sqrt{2p}}\exp\Big(-\frac{\psi_2}{16t}\Big)\Big\}\Big]\frac{t^2}{\psi_2}\exp\Big(-\frac{\psi_2}{32t}\Big).
	\end{align}
	and
	\begin{align}\label{2.3.8}
	II \leq 128 \Big[1 + \exp\Big\{ \frac{1}{\sqrt{p}} \exp\Big(-\frac{\psi_2}{32t}\Big)\Big\}\Big]\frac{t^2}{\psi_2} \exp\Big(-\frac{\psi_2}{64t}\Big).
	\end{align}
	
	Hence by combining \eqref{3.2.0}, \eqref{2.3.5}, \eqref{2.3.7} and \eqref{2.3.8}, we obtain for $0 < t \leq  \min\{1, \frac{\psi_2}{8\log(\psi_1B^6p)}\}$,
	\begin{align*}
	&\log\E\tr\exp\Big(t\sum_{j \in K_B}\Xb_j\Big)\\
	\leq &\log p + 4h(4)Bt^2\nu^2 + 151\Big[1 + \exp\Big\{ \frac{1}{\sqrt{p}} \exp\Big(-\frac{\psi_2}{64t}\Big)\Big\}\Big]\frac{t^2}{\psi_2} \exp\Big(-\frac{\psi_2}{64t}\Big).
	\end{align*}
	This completes the proof of Lemma \ref{lem:bernlm2}.
\end{proof}

\bibliographystyle{apalike}
\bibliography{MIAref}
\end{document}